%% file: 20171105Soberserialrewrite.tex
\documentclass{amsart}
\usepackage{amsthm}
\usepackage{amssymb}
\usepackage{amsmath}
\input xy
\xyoption{all} 

\usepackage{tikz}
\usetikzlibrary{snakes}
\newlength{\arrowsize}
\pgfarrowsdeclare{biggertip}{biggertip}{
  \setlength{\arrowsize}{0.4pt} 
  \addtolength{\arrowsize}{.5\pgflinewidth} 
  \pgfarrowsrightextend{0}
  \pgfarrowsleftextend{-5\arrowsize}
}{
  \setlength{\arrowsize}{0.8pt} 
  \addtolength{\arrowsize}{.5\pgflinewidth}
  \pgfpathmoveto{\pgfpoint{-5\arrowsize}{4\arrowsize}}
  \pgfpathlineto{\pgfpointorigin}
  \pgfpathlineto{\pgfpoint{-5\arrowsize}{-4\arrowsize}}
  \pgfusepathqstroke
}

\input{macros}

\newcommand{\Div}{\textnormal{Div}}

\newcommand\hfuzzReset{\hfuzz=3pt}
\hfuzzReset
\newcommand\toleranceReset{\tolerance=1400}
\toleranceReset
\newcommand\emergencystretchReset{\emergencystretch=2ex}
\emergencystretchReset
\title{Ziegler Spectra of Serial Rings}
\author{Lorna Gregory}
\address[Lorna Gregory]{University of Camerino, School of Science and Technologies, Division
of Mathematics, Via Madonna delle Carceri 9, 62032 Camerino, Italy}
\email{lorna.gregory@gmail.com}

\author{Gena Puninski}
\address[Gena Puninski]{Belarusian State University, Minsk, Belarus}

\thanks{The first author acknowledges the support of EPSRC through Grant EP/K022490/1. The second author died during the preparation of this article.}


\date{\today}

\begin{document}

\maketitle

\begin{abstract}
In this paper we prove that the Ziegler spectra of all serial rings are sober. We then use this proof to give a general framework for computing and understanding Ziegler spectra of uniserial rings up to topological indistinguishability. Finally, we illustrate this technique by computing the Ziegler spectra of all rank one uniserial domains up to topological indistinguishability.
\end{abstract}

The (right) Ziegler spectrum, $\Zg_R$, of a ring $R$ is a topological space attached to its module category of $R$. The points of $\Zg_R$ are isomorphism classes of indecomposable pure-injectives and the closed subsets correspond to complete theories of modules closed under arbitrary direct products.

In this article we investigate Ziegler spectra of serial rings.

The model theory of modules of serial rings was developed by Eklof and Herzog \cite{HerzogEklofserialrings} (see also \cite{Pun95})
using the Drozd-Warfield theorem on the structure of finitely presented modules.

In section \ref{soberserial}, we show, \ref{sober}, that the Ziegler spectrum of any serial ring is sober, i.e. each of its irreducible closed subsets is the closure of a point. Soberness of Ziegler spectra was first studied by Herzog in \cite{herzogduality} where he showed that every irreducible closed subset of $\Zg_R$ with a countable neighbourhood basis of open sets is the closure of a point. In particular, he showed that $\Zg_R$ is sober whenever $R$ is countable. His main motivation for showing that $\Zg_R$ is sober was to show that after identifying topologically indistinguishable points, the left and right Ziegler spectra of a ring are homeomorphic. Further, Ziegler spectra may be seen as analogous to (duals of) prime spectra of rings, see \cite[Section 14]{PSL} for details. Hochster, \cite{Hochsterspec}, showed that prime spectra of rings are exactly those topological spaces which are $\textbf{T}_0$, quasi-compact, have a basis of compact open sets which is stable under intersection and which are sober. We know that Ziegler spectra are quasi-compact with a basis of compact open sets. Their basis of compact open sets is rarely stable under intersection (for instance even the Ziegler spectrum of $\Z$ does not have this property) and they are in general not $\textbf{T}_0$. It is currently not known if there exist Ziegler spectra which are not sober.

Recently Gregory \cite{Lornasober} established that the Ziegler spectra for arbitrary commutative Pr\"{u}fer domains are sober. Her proof
is based on the corresponding result for commutative valuation domains, which is proved by brute force. The proof we give in this paper is much shorter and follows from the trichotomy theorem \ref{mainsoberclaim}.
Namely, if $T$ is the theory of an irreducible closed subset $C$ of the Ziegler spectrum of a serial ring $R$ and $e_1,\ldots,e_n$ are a complete set of primitive orthogonal idempotents for $R$, then either the theory $T$ of
$C$ contains a minimal pp-pair, or $C$ is the closure of a module realising the generic $e_i$-type of $T$ for some $1\leq i\leq n$ or the critical $e_i$-type of $T$ for some $1\leq i\leq n$. The generic $e_i$-type is the smallest pp-type of $T$ containing the formula $e_i|x$ and the critical $e_i$-type is the largest pp-type of $T$ containing the formula $e_i|x$.

Despite Pr\"{u}fer rings, that is, commutative rings with distributive ideal lattice, not being serial, as a corollary to our result for serial rings we get that all Ziegler spectra of Pr\"{u}fer rings are sober, see \ref{Prufersober}.

In section \ref{pointsforuniserial}, we turn our proof of soberness on its head and use it to give a general framework for computing and understanding Ziegler spectra of uniserial rings up to topological indistinguishability, equivalently up to elementary equivalence. That is, we describe all generic and critical pp-types of the theory of an indecomposable pure-injective module over a uniserial ring  and describe the indecomposable pp-types whose pure-injective hull is topologically distinguishable from the indecomposable pure-injective modules realising their generic and critical types \ref{pointsuptotopind}.

In section \ref{Sexamples}, we illustrate the techniques introduced in section \ref{pointsforuniserial} by computing the Ziegler spectra, after factoring out by $T_0$, of all rank one uniserial domains. These computations are based on the classification of Brungs and Dubrovin in \cite{B-D}.

A coarse version of the classification in \cite{B-D}, separates rank one uniserial domains into three classes: the nearly simple uniserial domains, that is, those with only one non-zero two-sided proper ideal; the invariant rank one uniserial domains, that is, those for which all right ideals are left ideals and visa-versa; and the exceptional rank one uniserial domains, that is, those with a prime ideal which is not completely prime. 

The case of nearly simple uniserial domains was covered in \cite{Punnearsim}. We include it here for completeness and to illustrate the techniques from section \ref{pointsforuniserial}. We also show, \ref{finiteideallattfinitezg}, that if $R$ is a uniserial ring with only finitely many two-sided ideals then $\Zg_R$ has only finitely many pairwise topologically indistinguishable points. We then go on to exhibit an example of a nearly simple uniserial domain whose Ziegler spectrum is finite.

Unsurprisingly, the invariant case turns out to be very similar to the case of commutative rank one valuation domains, in fact we show that for every invariant rank one uniserial domain $R$ there is a commutative rank one valuation domain $S$ such that $\Zg_R$ and $\Zg_S$ are homeomorphic. The techniques from section \ref{pointsforuniserial} don't play a role here. Instead, we attach a totally ordered abelian group $\Gamma_R$, the value group of $R$, to each invariant uniserial ring $R$. If $R$ is rank one then $\Gamma_R$ is rank one as a totally ordered abelian group and hence is a subgroup of $\R$, in particular $\Gamma_R$ is commutative. We then show that, as described in \cite{Lornasober} for the case of commutative valuation rings, the Ziegler spectrum of an invariant rank one uniserial domain $R$ can be described in terms of $\Gamma_R$. Finally, we give an explicit description of $\Zg_R$ up to topological indistinguishability for each $\Gamma_R$.

By far the hardest case to deal with is that of the exceptional rank one uniserial domains. Examples of such rings were constructed in \cite{B-D} using an embedding of the group ring of the universal covering of $\text{SL}_2(\R)$ into a division ring. From the point of view of classical ring theory, these examples are very hard to approach, due to their extreme noncommutativity. However, the model theory of modules has proved to be very useful in this situation; for instance, the lattice of pp-$1$-formulae carries information about left and right ideal structure of this ring, but also the way they are
interrelated.

The above model theoretic approach to the analysis of serial rings has already found many uses, say, in constructing counterexamples in
the theory of serial modules (see \cite{Punexcuni}, \cite{Punnearsim}), and in tilting theory (see \cite{BazHerPriSarTrl}). Due to
the rich source of highly nontrivial examples, we have no doubt that more applications are on the way.

\section{Preliminaries}

For general background on Model theory of Modules see \cite{Mikebook1} or for a more algebraic perspective see \cite{PSL}. For information about serial rings see \cite{Genaserialrings}.

Through out this paper we will work with right modules by default.

We start by recalling basic results about the model theory of modules.

Let $R$ be a ring. Let $\mcal{L}_R:=(0,+,(r)_{r\in R})$ be the language of (right) $R$-modules. A (right) \textbf{pp-n-formula} is a formula of the form \[\exists \overline{y} \ (\overline{y},\overline{x})A=0\] where $l,n,m$ are natural numbers, $A$ is an $(l+n)\times m$ matrix with entries from $R$, $\overline{y}$ is an $l$-tuple of variables and $\overline{x}$ is an $n$-tuple of variables.

 We write $\varphi(M)$ for the solution set of a pp-formula $\varphi$ in an $R$-module $M$. For any pp-$n$-formula $\phi$ and $R$-module $M$, $\phi(M)$ is a subgroup of $M^n$ under $+$ and if $\phi$ is a pp-$1$-formula then $\phi(M)$ is a left $\End(M)$-submodule of $M$.

 If we weaken our definition of a pp-formula to include all formulae (in one variable) in the language of (right) $R$-modules, $\mcal{L}_R$, which are equivalent over the theory of $R$-modules, $T_R$, to a pp-formula then the $T_R$-equivalence classes of pp-$n$-formulae become a lattice under implication with the join of two formulae $\phi,\psi$ given by
\[(\phi+\psi)(x):=\exists y,z(x=y+z\and\phi(y)\and\psi(z))\] and the meet given by $\phi\and\psi$.

A pp-pair, written $\phi/\psi$, is a pair of pp-formulae $\phi,\psi$ such that $\phi(M)\supseteq \psi(M)$ for all $R$-modules $M$. If $\psi$ does not imply $\phi$ then we identify $\phi/\psi$ with $\phi/\phi\and\psi$. We write $[\psi,\phi]$ for the interval in $\pp_R^n$. If $M$ is an $R$-module and $\phi/\psi$ is a pp-pair then we say that $M$ \textbf{opens} $\phi/\psi$ if $\phi(M)\supsetneq \psi(M)$.

\begin{definition}
Let $R$ be a ring. An \textbf{invariants sentence} is a sentence in $\mcal{L}_R$ which expresses the statement $\vertl\frac{\phi(\bar{x})}{\psi(\bar{x})}\vertr\geq n$ in all modules, for some $\phi,\psi$ pp-formulae of the same arity and $n\in \mathbb{N}$.
\end{definition}

\begin{theorem}[Baur-Monk Theorem]\cite{Mikebook1}
Let $R$ be a ring. Every formula in $\mcal{L}_R$ is equivalent over $T_R$ to a boolean combination of pp-formulae and invariants sentences.
\end{theorem}

\begin{cor}\cite[2.18]{Mikebook1}
Let $R$ be a ring and $M,N$ be $R$-modules. Then $M$ is elementary equivalent to $N$ if and only if for all pp-$1$-pairs $\phi/\psi$, \[\vertl\phi(M)/\psi(M)\vertr=\vertl\phi(N)/\psi(N)\vertr\] whenever either side of the equality is finite.
\end{cor}

A \textbf{pure-embedding} between two modules is an embedding which reflects the solution sets of pp-formulae. We say a module $N$ is \textbf{pure-injective} if for every pure-embedding $g:N\rightarrow M$, the image of $N$ in $M$ is a direct summand of $M$. We write $\pinj_R$ for the set of isomorphism classes of indecomposable pure-injective modules.

The \textbf{Ziegler spectrum} of a ring $R$, denoted $\zg_R$, is a topological space whose points are isomorphism classes of indecomposable pure-injective modules and which has a basis of open sets given by:
\[(\phi/\psi)=\{M\in\pinj_R \st \phi(M)\supsetneq\psi(M)\and\phi(M)\}\]
where $\varphi,\psi$ range over pp-$1$-formulae.

The sets $(\phi/\psi)$ are compact, in particular, $\Zg_R$ is compact.

The closed subsets of $\Zg_R$ correspond to theories of modules closed under arbitrary products and direct summands via the following maps

\[C \mapsto T(C)=\text{Th}\{N_i^{(\aleph_0)} \st N_i\in C\}\] and
\[T\mapsto \{N\in\Zg_R \st N^{(\aleph_0)} \text{ is a model of } T\}.\]

Given a closed subset $C$ of $\Zg_R$, we write $\pp_R^nC$ for the lattice of pp-$n$-formulae after factoring out by the equivalence relation given by setting $\phi$ equivalent to $\psi$ if and only if $\phi(N)=\psi(N)$ for all $N\in C$. We write $\phi\geq_{C}\psi$ for the order on this lattice. If $\phi/\psi$ is a pp-pair then we write $[\psi,\phi]_C$ for the interval in $\pp_R^nC$.

We say that a pair of pp-formulae $\phi/\psi$ is a $C$-\textbf{minimal pair} if there is no pp-formula $\sigma$ such that $\phi>_{C}\sigma>_{C}\psi$.

A \textbf{pp-type} is a filter in $\pp_R^n$. If $M$ is an $R$-module and $\overline{a}\in M$ then the set of pp-formulae satisfied by $\overline{a}$ in $M$ is called the \textbf{pp-type of} $\overline{a}$. Conversely, if $\overline{a}$ is an $n$-tuple from a module $M$ with pp-type $p$ then we say that $\overline{a}$ \textbf{realises} $p$. For every pp-type $p$ there exists a pure-injective module $M$ and $\overline{a}\in M$ such that $\overline{a}$ realises $p$ and such that if $\overline{a}\in M'\subseteq M$ and $M'$ is pure-injective then $M'=M$. The module $M$ is determined up to isomorphism by $p$ and we call it the \textbf{pure-injective hull} of $p$ and write $N(p)$. See \cite[3.6]{Zieglermodules}, \cite[Chapter 4]{Mikebook1} and for a different perspective \cite[Section 4.3.5]{PSL}.

A pp-type is called \textbf{irreducible} if it is realised in an indecomposable pure-injective module.

\begin{lemma}[Ziegler's irreducibility criterion]\cite[4.4]{Zieglermodules}
A pp-type $p$ is irreducible if and only if for all $\phi_1,\phi_2\notin p$ there exists $\sigma\in p$ such that $\phi_1\wedge\sigma+\phi_2\wedge\sigma\notin p$.
\end{lemma}

We say a pp-pair $\phi/\psi$ is contained in a pp-type $p$, and write $\phi/\psi\in p$, if $\phi\in p$ and $\psi\notin p$.

\begin{lemma}\cite[7.10]{Zieglermodules}\label{distinguishingnonisopinj}
Let $p,q$ be irreducible pp-types containing $\phi/\psi$. If $N(p)$ is not isomorphic to $N(q)$ then there exists a pp-formula $\chi$ such that $\psi\subseteq\chi\subseteq\phi$ and either $\phi/\chi\in p$ and $\chi/\psi\in q$, or $\phi/\chi\in q$ and $\chi/\psi\in p$.
\end{lemma}

Two points $x,y$ in a topological space $T$ are \textbf{topologically indistinguishable} if for all open sets $U$ of $T$, $x\in U$ if and only if $y\in U$. We call a point in a topological space a \textbf{$T_0$-point} if it is topologically distinguishable from all other points in the space. If $T$ is a topological space then we write $T/T_0$ for the topological space with underlying set the equivalence classes of topologically indistinguishable points and the quotient topology induced by $T$.

In this paper, we will often study $\Zg_R/T_0$ rather than $\Zg_R$. The following proposition explains why this is a reasonable thing to do. It is likely to be well known but we include a proof since we were not able to find a reference.

\begin{proposition}
If $N,M\in\pinj_R$ are topologically indistinguishable as points in the Ziegler spectrum of $R$ then $N$ and $M$ are elementary equivalent as $R$-modules.
\end{proposition}
\begin{proof}
First note that, by the Baur-Monk theorem, if for all pp-$1$-pairs $\phi/\psi$ we have that either $|\phi(N)/\psi(N)|=1$ or $|\phi(N)/\psi(N)|$ is infinite, and $|\phi(M)/\psi(M)|=1$ or $|\phi(M)/\psi(M)|$ is infinite then if $N$ and $M$ are topologically indistinguishable then they are elementarily equivalent.

Suppose that there exists a pp-$1$-pair $\phi/\psi$ such that $|\phi(N)/\psi(N)|$ is finite but not equal to $1$. Then there exists $\psi\leq\tau<\phi$ such that $\phi/\tau$ is an $N$-minimal pair. Suppose, for a contradiction, that $M$ and $N$ are topologically indistinguishable but not isomorphic. Let $p$ be the pp-type of an element $a$ of $N$ such $a\in\phi(N)$ and $a\notin\tau(N)$. Let $q$ be the pp-type of an element $b$ of $M$ such that $b\in\phi(M)$ and $b\notin\tau(M)$. By \ref{distinguishingnonisopinj}, there exists $\phi\supseteq\chi\supseteq\tau$ such that either $\phi/\chi\in p$ and $\chi/\tau\in q$, or $\phi/\chi\in q$ and $\chi/\tau\in p$. In the first case $\chi(N)=\tau(N)$, so $|\chi(N)/\tau(N)|= 1$, and $\chi/\tau\in q$, so $|\chi(N)/\tau(N)|>1$. In the second case, $\chi(N)=\phi(N)$ so $|\phi(N)/\chi(N)|=1$, and $\phi/\chi\in q$, so $|\phi(M)/\chi(M)|>1$. Thus $N$ and $M$ are not topologically indistinguishable.

\end{proof}

We now specialise to the case of serial rings.

A module $M$ is said to be \textbf{uniserial} if its lattice of submodules is totally ordered by inclusion and said to be \textbf{serial} if it is a direct sum of uniserial modules. A ring $R$ is called \textbf{uniserial} (\textbf{serial}) if it is uniserial (serial) as both a right and left module over itself.

For any serial ring there exists orthogonal idempotents $e_1,\ldots, e_n$ such that $1=e_1+\ldots+e_n$ and, $e_iR$ and $Re_i$ are uniserial (and hence indecomposable). We call an idempotent \textbf{primitive} if $eR$ is uniserial and a set of primitive orthogonal idempotents $e_1,\ldots,e_n$ such that $e_1+\ldots+e_n=1$ a \textbf{complete} set of primitive orthogonal idempotents.

Let $R$ be a serial ring and $e\in R$ a primitive idempotent. We call a pp-$1$-formula $\phi(x)$ an \textbf{$e$-formula} if $\phi(x)\rightarrow e|x$, where $e|x$ denotes the pp-formula $\exists y \ x=ye$.

\begin{lemma}\cite[Corollary 1.6]{HerzogEklofserialrings}\cite[Lemma 11.1]{Genaserialrings}\label{descpp1flaoverserial}
Let $R$ be a serial ring and $e_1,...,e_n$ a complete set of primitive orthogonal idempotents. Every pp-$1$-formula over $R$ is equivalent to a finite sum of pp-formulae of the form $s|x\wedge xt=0$, where $s\in e_iR$ and $t\in Re_j$. Moreover, if $\phi$ is an $e$-formula then we may assume that $s\in e_iRe$ and $t\in eRe_j$.
\end{lemma}

An $e$-pair is a pair $\langle I,J\rangle$, where $I\subseteq eR$ is a right ideal and $J\subseteq Re$ is a left ideal of $R$.

\begin{lemma}\cite[Lemma 11.2]{Genaserialrings}\label{efladistr}
Let $R$ be a serial ring and $e$ a primitive idempotent. The lattice of all $e$-formulae over $R$ is distributive.
\end{lemma}

For an $e$-pair $\langle I,J\rangle$ we define a collection of pp-formulae and negations of pp-formulae
\[J^*/I:=\{s|x\st s\notin J\}\cup\{xr=0 \st r\in I\}\cup \{\neg (sr|xr) \st s\in J, r\notin I\}.\]

\begin{lemma}\cite[11.8]{Genaserialrings}\cite[Theorem 2.7]{HerzogEklofserialrings}
Let $R$ be a serial ring,  $e\in R$ a primitive idempotent and $\langle I,J\rangle$ an $e$-pair. If the set of formulae $J^*/I$ is consistent then it has a unique extension to an irreducible $e$-type over $R$, that is an irreducible pp-$1$-type containing the formula $e|x$. Moreover, all irreducible $e$-types over $R$ are obtained in this way.
\end{lemma}

We will write $N(I,J)$ for the unique indecomposable pure-injective module realising $J^*/I$. We say that an $e$-pair $\langle I,J\rangle$ is consistent if $J^*/I$ is consistent. Note that after fixing a complete set of primitive orthogonal idempotents $e_1,\ldots, e_n$, the above lemma implies that all indecomposable pure-injectives are of the form $N(I,J)$ for some consistent $e_i$-pair.

\begin{lemma}\label{Genasconsistencycondition}\cite[11.9]{Genaserialrings}\cite[4.5]{supdecoverserial}
Let $e$ be a primitive idempotent of a serial ring $R$. For a (right) $e$-pair $\langle I,J\rangle$ the following are equivalent:
\begin{enumerate}
\item $\langle I,J\rangle$ is consistent
\item for all $r\in I$, $r^*\notin I$, $s\in J$ and $s^*\notin J$, $s^*r^*\neq sr^*+s^*r$
\item for all $r\in I$, $r^*\notin I$, $s\in J$ and $s^*\notin J$, $s^*r^*\notin RsrR$.
\end{enumerate}
\end{lemma}

\begin{lemma}\cite[2.10]{HerzogEklofserialrings}\cite[11.11]{Genaserialrings}\label{equonepairs}
Let $R$ be a serial ring. Let $\langle I_1,J_1\rangle$ be a consistent $e$-pair and $\langle I_2,J_2\rangle$ be a consistent $f$-pair for primitive idempotents $e,f\in R$. The indecomposable pure-injectives $N(I_1,J_1)$ and $N(I_2,J_2)$ are isomorphic if and only if there exist $u\in eRf$ and $v\in fRe$ such that either
\begin{enumerate}
\item $u\neq 0$, $I_1=uI_2$ and $J_1u=J_2$ or
\item $v\neq 0$, $vI_1=I_2$ and $J_1=J_2v$.
\end{enumerate}
\end{lemma}

\section{Serial rings}\label{soberserial}

In this section we will show that for any serial ring $R$, $\Zg_R$ is sober. That is, we will show that for every irreducible closed subset $C$, there exists $N\in C$ such that $\cl N=C$. Equivalently, we will show that for all irreducible closed subsets $C$, there exists an $N\in C$ such that for all open sets $U$ such that $U\cap C\neq \emptyset$, $N\in U$. We call such a point a generic point of $C$. Since Ziegler spectra are not always $T_0$, generic points are not necessarily unique.

\begin{lemma}\label{distimpliesirredclosedsetischain}
Let $R$ be a ring, $C\subseteq \Zg_R$ be an irreducible closed subset and $\phi/\psi$ a pp-$1$-pair. If the interval $[\psi,\phi]$ is distributive then $[\psi,\phi]_C$ is a totally ordered.
\end{lemma}
\begin{proof}
By \cite[Theorem 1.21]{Tuganbaevsemidistringsmods} every distributive module over a local ring is uniserial. The endomorphism rings of indecomposable pure-injective modules are local. Thus, for any indecomposable pure-injective module $N$, $\phi(N)/\psi(N)$ is uniserial as a module over $\End(N)$. So for every indecomposable pure-injective module $N$, $[\psi,\phi]_N$ is totally ordered.  Therefore, if $\sigma,\tau\in [\psi,\phi]$ then
\[[\Zg_R\backslash\left(\sigma/\sigma\wedge\tau\right)]\cup[\Zg_R\backslash\left(\tau/\sigma\wedge\tau\right)]=\Zg_R.\]

Hence for every irreducible closed subset of $\Zg_R$, either $C\subseteq\Zg_R\backslash\left(\sigma/\sigma\wedge\tau\right)$ or $C\subseteq\Zg_R\backslash\left(\tau/\sigma\wedge\tau\right)$. Thus $[\psi,\phi]_C$ is totally ordered.

\end{proof}

\begin{lemma}\cite[Corollary 2.5]{HerzogEklofserialrings}\label{euniserial}
Let $R$ be a serial ring and $e_1,...,e_n$ a complete set of primitive orthogonal idempotents. Every pure-injective indecomposable module $N$ considered as an $\End N$-module has a direct sum decomposition $N=\oplus_{i=1}^n Ne_i$ into uniserial $\End N$-submodules $Ne_i$.
\end{lemma}

It follows from the above lemma that for a serial ring $R$ with $e_1,...,e_n$ a complete set of primitive orthogonal idempotents, we have
\[\left(e_1|x/x=0\right)\cup...\cup\left(e_n|x/x=0\right)=\Zg_R.\]

\begin{cor}\label{irredclosedsetchainserial}
Let $R$ be a serial ring and $e$ a primitive idempotent. Let $C\subseteq \Zg_R$ be an irreducible closed subset. Then the interval $[x=0,e|x]_C$ is totally ordered. In particular, if $N$ is indecomposable pure-injective then the subsets defined by $e$-formulae are totally ordered by inclusion.
\end{cor}
\begin{proof}
By \ref{efladistr}, the interval $[x=0,e|x]$ is distributive. So, by \ref{distimpliesirredclosedsetischain}, $[x=0,e|x]_C$ is totally ordered.
\end{proof}

Let $R$ be a serial ring, $C$ an irreducible closed subset of $\Zg_R$. For each primitive idempotent $e\in R$ such that $\left(e|x/x=0\right)\cap C\neq \emptyset$, we define two important irreducible pp-types $\wp_{e-\text{gen}}$ and $\wp_{e-\text{crit}}$.

Any (irreducible) filter on the interval $[x=0,e|x]_C$ may be extended, using \cite[4.33]{Mikebook1}, to an irreducible pp-type and although there may be many different extensions, they all have the same pure-injective hull by \cite[7.10]{Zieglermodules}. Thus we define the \textbf{generic} pp-type $\wp_{e-\text{gen}}$ to be the filter on $[x=0,e|x]_C$ containing $e|x$ but not containing any formula in this interval strictly below $e|x$ (or more precisely the pre-image in $[x=0,e|x]$ of this filter). We denote the pure-injective hull of this type by $M_{e-\text{gen}}$.

Similarly, we define the \textbf{critical} pp-type $\wp_{e-\text{crit}}$ as the filter on $[x=0,e|x]_C$ containing all formulae in the interval strictly above $x=0$. We denote the pure-injective hull of this type by $M_{e-\text{crit}}$.

Note that if $C$ is an irreducible closed subset then $M_{e-\text{gen}}, M_{e-\text{crit}}\in C$. This is because if $M_{e-\text{gen}}\notin C$ then, by \cite[4.9]{Zieglermodules} there exists $\phi/\psi$ such that $\phi\in \wp_{e-\text{gen}}$ and $\psi\notin \wp_{e-\text{gen}}$ and $\left(\phi/\psi\right)\cap C=\emptyset$. Note that $M_{e-\text{gen}}\in\left(\phi\wedge e|x/\psi\wedge e|x\right)\subseteq \left(\phi/\psi\right)$ and by definition $\left(\phi\wedge e|x/\psi\wedge e|x\right)\cap C\neq \emptyset$. Exactly the same argument works for $M_{e-\text{crit}}$. When we are not working in a fixed irreducible closed subset $C$ we will write $\wp^{C}_{e-\text{gen}}$ and $\wp^{C}_{e-\text{crit}}$ to indicate that these are the generic and critical types of $C$.

We need the following auxiliary result.

\begin{lemma}\label{div}
Let $R$ be a serial ring and  $e_1,e_2\in R$ be primitive idempotents.
\begin{enumerate}
\item Suppose that $\psi<_{C} r\mid x$ for $r\in e_1Re_2$ and some $e_2$-formula $\psi$. Then $M_{e_1-\text{gen}}$ opens the pair
$r\mid x/\psi$.
\item Suppose that $e_1\mid x\wedge xr=0<_C \phi$ for $r\in e_1Re_2$ and some $e_1$-formula $\phi$. Then $M_{e_2-\text{crit}}$ opens
the pair $\phi/e_1\mid x\wedge xr=0$.
\end{enumerate}
\end{lemma}
\begin{proof}
Goursat's theorem, \cite[8.9]{Zieglermodules},  states that any $\rho\in \pp_R^2$ defines a lattice isomorphism between the (possibly trivial) intervals $[\rho(x,0),\exists y \ \rho(x,y)]$ and $[\rho(0,y),\exists x \ \rho(x,y)]$ defined by
\[\phi(x)\in [\rho(x,0),\exists y \ \rho(x,y)]\mapsto\exists x \ \phi(x)\wedge\rho(x,y)\in [\rho(0,y),\exists x \ \rho(x,y)].\] The inverse of this isomorphism sends $\psi(y)\in [\rho(0,y),\exists x \ \rho(x,y)]$ to $\exists y \ \psi(y)\wedge\rho(x,y)$.

(1) Goursat's theorem with $\rho(x,y)$ equal to $x=yr\wedge e_1|y$ implies that for any module $N$, $N$ opens $r\mid x/\psi$ if and only if $N$ opens $e_1|y/\exists x (\psi(x)\wedge x=yr\wedge e_1|y)$. So if $\psi<_{C} r\mid x$ then by definition $M_{e_1-\text{gen}}$ opens $e_1|y/\exists x (\psi(x)\wedge x=yr\wedge e_1|y)$. Thus $M_{e_1-\text{gen}}$ opens $r\mid x/\psi$.

(2) Again using Goursat's theorem with $\rho(x,y)$ equal to $x=yr\wedge e_1|y$, we get that for any module $N$, $N$ opens $\phi(y)/e_1|y\wedge yr=0$ if and only if $N$ opens $\exists y (\phi(y)\wedge x=yr\wedge e_1|y)/x=0$. Noting that $\exists y (\phi(y)\wedge x=yr\wedge e_1|y)$ is an $e_2$-formula since $r\in e_1Re_2$, the proof is now as in (1).


\end{proof}

The following lemma is an easy exercise in Ziegler topology.

\begin{lemma}\label{open}
Let $R$ be a ring and $C$ be an irreducible closed subset of $\Zg_R$. Suppose that
$\phi/\psi$ is a pp-pair such that $(\phi/\psi)\cap C\neq \emptyset$. Then $M\in C$ is a generic point in $C$ if and only if it opens
each pair $\phi'/\psi'$ such that $\psi\leq \psi'<_C \phi'\leq \phi$.

In particular, if there is a $C$-minimal pair, then the unique indecomposable pure-injective opening this pair is
the generic point.
\end{lemma}
\begin{proof}
The forward direction is a direct consequence of the definition of generic point. The reverse direction follows from \cite[4.9]{Zieglermodules}.
\end{proof}

Now we deal with the main claim.

\begin{proposition}\label{mainsoberclaim}
Let $e_1,\ldots,e_n$ be a complete set of orthogonal indecomposable idempotents of a serial ring $R$. Let $C$ be an irreducible closed subset of $\Zg_R$ and suppose that $\left(e_i|x/x=0\right)\cap C\neq \emptyset$ for $1\leq i\leq m$ and $\left(e_i|x/x=0\right)\cap C= \emptyset$ for $m+1\leq i\leq n$.
 Then at least one of the following holds.
\begin{enumerate}
\item The interval $[x=0, e_i\mid x]$ has a $C$-minimal pair for some $1\leq i\leq m$.
\item There exists $1\leq i\leq m$ such that $C$ is the closure of $M_{e_i-\text{gen}}$.
\item There exists $1\leq i\leq m$ such that $C$ is the closure of $M_{e_i-\text{crit}}$.
\end{enumerate}
\end{proposition}
\begin{proof}

We know that the interval $[x=0, e_1| x]_C$ is totally ordered, hence each pp-formula in it is equivalent to either a
divisibility formula $s|x$, for some $s\in e_iRe_1$ with $1\leq i\leq m$, or to an annihilator formula $xr=0\wedge e_1| x$ for some $r\in e_1Re_i$ with $1\leq i\leq m$.

Suppose (1) does not hold. We will show that every pair $\phi/\psi$ of $e_1$-formulae which is opened by some module in $C$ is opened by $M_{e_i-\text{gen}}$ or $M_{e_i-\text{crit}}$ for some $1\leq i\leq m$. Since $\phi/\psi$ is not a $C$-minimal pair there exists $\sigma$ such that $\phi>_{C}\sigma>_{C}\psi$.

If $\sigma$ is equivalent to $s\mid x$ for $s\in e_iRe_1$ then $M_{e_i-\text{gen}}$ opens $\sigma/\psi$, by \ref{div}, and hence opens $\phi/\psi$.

If $\sigma$ is equivalent to $xr=0\wedge e_1\mid x$ for $r\in e_1Re_i$ then $M_{e_i-\text{crit}}$ opens the pair $\phi/\sigma$, by \ref{div}, and hence opens $\phi/\psi$.

Thus
\[C=\bigcup_{i=1}^m\text{cl}M_{e_i-\text{gen}}\cup\bigcup_{i=1}^m \text{cl}M_{e_i-\text{crit}}.\]

Since $C$ is irreducible, either $C=\text{cl}M_{e_i-\text{gen}}$ or $C=\text{cl}M_{e_i-\text{crit}}$ for some $1\leq i\leq m$.

\end{proof}

If $N\in\Zg_R$ then write $N_{e_i-\text{gen}}$ for $M^{\cl N}_{e_i-\text{gen}}$ and $N_{e_i-\text{crit}}$ for $M^{\cl N}_{e_i-\text{crit}}$.

\begin{cor}\label{mainclosureclaim}
Let $e_1,\ldots,e_n$ be a complete set of orthogonal indecomposable idempotents of a serial ring $R$. Let $N$ be an indecomposable pure-injective $R$-module and  suppose that $Ne_i\neq 0$ for $1\leq i\leq m$ and $Ne_i=0$ for $m+1\leq i\leq n$.
 Then at least one of the following holds.
\begin{enumerate}
\item $N$ has a minimal pair in the interval $[x=0,e_i]$ for some $1\leq i\leq m$
\item There exists $1\leq i\leq m$ such that $N$ is in the closure of $N_{e_i-\text{gen}}$ and hence $N\equiv N_{e_i-\text{gen}}$.
\item There exists $1\leq i\leq m$ such that $N$ is in the closure of $N_{e_i-\text{crit}}$ and hence $N\equiv N_{e_i-\text{crit}}$.
\end{enumerate}
\end{cor}

Now we are in a position to prove the main result of this section.

\begin{theorem}\label{sober}
Let $R$ be a serial ring. The Ziegler spectrum of $R$ is sober.
\end{theorem}
\begin{proof}
Let $e_1,\ldots,e_n$ be a complete set of orthogonal indecomposable idempotents for $R$. Suppose that $C$ is an irreducible closed subset of $\Zg_R$ and suppose that $\left(x=x/xe_i=0\right)\cap C\neq \emptyset$ for $1\leq i\leq m$ and $\left(x=x/xe_i=0\right)\cap C= \emptyset$ for $m+1\leq i\leq n$. Then by \ref{mainsoberclaim}, either $[x=0,e_1|x]$ has a $C$-minimal pair or $C$ is the closure of $M_{e_i-\text{gen}}$ or $M_{e_i-\text{crit}}$ for some $1\leq i\leq m$.
If $C$ has a minimal pair then by \ref{open}, $C$ is the closure of the unique point opening that pair.
\end{proof}

\begin{question}
What is the relationship between the modules $M_{e_i-\text{gen}}$ (respectively $M_{e_i-\text{crit}}$) as $i$ varies?
\end{question}

We say that a commutative ring is a \textbf{valuation ring} if its lattice of ideals is totally ordered and a \textbf{Pr\"{u}fer ring} if its localisation at each of its maximal ideals is a valuation ring. Equivalently, a commutative ring is a Pr\"{u}fer ring if its ideal lattice is distributive \cite[Theorem 1]{Arithrings}. Eklof and Herzog showed that the Pr\"{u}fer rings are exactly those commutative rings with distributive lattice of pp-$1$-formulas \cite[3.1]{HerzogEklofserialrings}.

\begin{proposition}\cite[6.5]{Lornasober}
Let $R$ be a commutative ring. The following are equivalent:
\begin{enumerate}[(i)]
\item $\Zg_R$ is sober,
\item for all $\mfrak{p}\lhd R$ prime, $\Zg_{R_{\mfrak{p}}}$ is sober,
\item for all $\mfrak{m}\lhd R$ maximal, $\Zg_{R_{\mfrak{m}}}$ is sober.
\end{enumerate}
\end{proposition}

Since all valuation rings are trivially serial, we get the following corollary.

\begin{cor}\label{Prufersober}
Let $R$ be a Pr\"{u}fer ring. The Ziegler spectrum of $R$ is sober.
\end{cor}

\section{Uniserial rings}\label{pointsforuniserial}

Throughout this section, $R$ will be a uniserial ring. The (unique) maximal right ideal of $R$ is equal to the Jacobson radical $\J(R)$ of $R$, which is equal to the (unique) maximal left ideal of $R$. Thus, $\J(R)$ is the set of non-units of $R$.

Over a uniserial ring, each indecomposable pure-injective module $N$ has a unique generic and critical pp-type namely
\[\wp^{N}_\text{gen}:=\{\phi\in \pp_R^1\st \phi(N)=N\}\] and
\[\wp^N_{\text{crit}}:\{\phi\in\pp_R^1 \st \phi(N)\neq \{0\}\}.\]

These types are exactly $\wp^{C}_{1-\text{gen}}$ and $\wp^{C}_{1-\text{crit}}$ where $C$ is the Ziegler closure of $N$. We write $N_{\text{gen}}$ (respectively $N_{\text{crit}}$) for the indecomposable pure-injective module realising $\wp^{N}_\text{gen}$ (respectively $\wp^N_{\text{crit}}$).

\begin{lemma}\label{uniserialringsgenandcrit}
Let $R$ be a uniserial ring, $\phi\in\pp_R^1$ and $N$ an indecomposable pure-injective.  If $N\in\left(a|x/\phi\right)$ then $N_{\text{gen}}\in \left(a|x/\phi\right)$. If $N\in\left(\phi/xb=0\right)$ then $N_{\text{crit}}\in \left(\phi/xb=0\right)$.
\end{lemma}
\begin{proof}
This is a direct consequence of \ref{div}.
%
\end{proof}

\begin{lemma}
Let $N$ be an indecomposable pure-injective over a uniserial ring $R$. The Ziegler closure of $N$ is equal to $N$ together with the Ziegler closures of $N_{\text{gen}}$ and $N_{\text{crit}}$.
\end{lemma}
\begin{proof}
Corollary \ref{mainclosureclaim} specialised to uniserial rings implies that either $N$ is in the closure of $N_{\text{gen}}$, $N$ is in the closure of $N_{\text{crit}}$ or $N$ has a minimal pair and is hence isolated in its closure. So suppose that $N$ is isolated in its closure.

We will now show that for any open set $U$ with $N\in U$, either $N_{\text{crit}}\in U$, $N_{\text{gen}}\in U$ or $U$ isolates $N$ in its closure. It is enough to show that this condition holds for basic open sets and thus we may assume $U$ is given by a pair of pp-$1$-formulae. Suppose that $N\in\left(\phi/\psi\right)$ but $\left(\phi/\psi\right)$ does not isolate $N$ in its closure. Thus there exists $\sigma$ such that $\phi(N)>\sigma(N)>\psi(N)$. Since $R$ is uniserial, working with respect to the theory of $N$ we may assume that $\sigma$ is either of the form $r|x$ or $xs=0$. If $\sigma$ is of the form $r|x$ then $N_{gen}$ opens the pair $\sigma/\psi$ and if $\sigma$ is of the form $xs=0$ then $N_{crit}$ opens the pair $\phi/\sigma$. Thus for all open sets $U$, either $N_{gen}\in U$, $N_{crit}\in U$ or $U$ isolates $N$ in its closure.

Thus the open set $U:=\Zg_R\backslash\text{cl}\{N_{\text{gen}},N_{\text{crit}}\}$ isolates $N$ in its closure. So, if $M\in\text{cl}\{N\}$ then either $M\in \text{cl}\{N_{\text{gen}},N_{\text{crit}}\}$ or $M\in U$ and thus $N=M$.

\end{proof}

\begin{cor}
If $N,M\in\Zg_R$ are not isolated in their closures, $N_{\text{gen}}=M_{\text{gen}}$ and $N_{\text{crit}}=M_{\text{crit}}$ then $N\equiv M$.
\end{cor}
\begin{proof}
This is true because if $M\in\Zg_R$ is not isolated in its closure then $M\equiv M_{\text{gen}}$ or $M\equiv M_{\text{crit}}$.
\end{proof}

\begin{definition}
Let $N\in\Zg_R$.
\begin{enumerate}
\item $\Ass N:=\{r\in R \st mr=0 \text{ for some } 0\neq m\in N\}$
\item $\Div N:=\{r\in R \st r \text{ does not divide } m \text{ for some } m \in N\}$
\item $\ann_RN:=\{r\in R \st mr=0 \text{  for all }m\in N\}$
\end{enumerate}
\end{definition}

A right ideal $P$ is \textbf{completely prime} if $xy\in P$ implies $x\in P$ or $y\in P$.

Note that for any $N\in\Zg_R$, $\Ass N$ and $\Div N$ are completely prime two-sided ideals and $\ann_RN$ is a two-sided ideal. Further, note that $rs\in \ann_RN$ implies $r\in \ann_R N$ or $s\in \Ass N$, and, $s\in \ann_RN$ or $r\in \Div N$.

\begin{lemma}
The generic type of $N\in \Zg_R$ is $\langle\ann_R N,\Div N\rangle$ and the critical type of $N\in\Zg_R$ is $\langle\Ass N,\ann_RN\rangle$.
\end{lemma}
\begin{proof}
By definition $a|x<x=x$ in $N$ if and only if $a\in\Div N$ and $xb=0$ is equivalent to $x=x$ in $N$ if and only if $b\in \ann_RN$. By definition $xa=0>x=0$ in $N$ if and only if $a\in \Ass N$ and $b|x$ is equivalent to $x=0$ if and only if $b\in\ann_RN$.
\end{proof}

\begin{lemma}\label{desccrit}
Let $P$ be a completely prime two-sided ideal and $K$ a two-sided ideal of $R$. If
\begin{enumerate}
\item $rs\in K$ implies $r\in K$ or $s\in P$
\end{enumerate}
then $\langle P,K\rangle$ is consistent and $\langle P,K\rangle$ is the critical type of $N(P,K)$.
\end{lemma}
\begin{proof}
Suppose that $c\in K$ and $b\in P$. Then $RcbR\subseteq K$ since $K$ is two-sided. Thus by (1), $rs\in RcbR$ implies $r\in K$ or $s\in P$. Therefore, by \ref{Genasconsistencycondition}, $\langle P,K\rangle$ is consistent.

We first show that $\Ass(N(P,K))=P$.

If $I$ is a right ideal, $u\neq 0$ and $uI=P$ then, since $P$ is completely prime, either $u\in P$ or $P\supseteq I$. If $u\in P$ then $uI\subseteq u\J(R)\subsetneq uR\subseteq P$. Thus $P\supseteq I$.
So $uI=P$ implies $u\notin P$ and $P\supseteq I$. If $I$ is a right ideal and $I=uP$ then $I=uP\subseteq P$ since $P$ is two-sided. Therefore, by \ref{equonepairs}, $\Ass(N(P,K))=P$.

We now show that $a\in\ann_RN(P,K)$ if and only if $a\in K$. Note that if $a\in\ann_RN(P,K)$ then $a$ does not divide any element of $N(P,K)$. Thus $K\supseteq \ann_RN(P,K)$.

First suppose that $\langle I,J\rangle$ is a consistent pair, $uI=P$ and $J=Ku$. We have shown above that this implies that $u\notin P$.

If $a\in K$ then $a\in K\subseteq P\subseteq Ru$. Thus $a=\lambda u$ for some $\lambda\in R$. By condition (1) this implies $\lambda\in K$. So $a\in Ku=J$.

Now suppose that $\langle I,J\rangle$ is a consistent pair, $I=uP$ and $Ju=K$. Note that $u\notin I$ since if it were then $uP\subsetneq uR\subseteq I$. If $a\in K$ then $au\in K$ since $K$ is two-sided. Thus there exists $j\in J$ such that $au=ju$. Thus $(a-j)u=0$ and by \ref{Genasconsistencycondition}, either $a-j\in J$ or $u\in I$. Since $u\notin I$, $a-j\in J$ and hence $a\in J$.

Thus, by \ref{equonepairs}, $a\in K$ implies $a\in \ann_RN(P,K)$. So $\ann_RN(P,K)=K$.

\end{proof}

\begin{lemma}\label{descgen}
Let $P$ be a completely prime two-sided ideal and $K$ a two-sided ideal of $R$. If
\begin{enumerate}
\item $sr\in K$ implies $r\in K$ or $s\in P$
\end{enumerate}
then $\langle K,P\rangle$ is consistent and $\langle K,P\rangle$ is the generic type of $N(K,P)$.
\end{lemma}
\begin{proof}
The proof of this lemma follows from the previous lemma using Herzog's duality \cite{herzogduality}. Alternatively, it can be proved as the previous lemma making appropriate modifications.
\end{proof}

Now that we have good descriptions of the generic and critical types of modules in $\Zg_R$, we investigate those $N\in \Zg_R$ which are not elementary equivalent to $N_{\text{gen}}$ or $N_{\text{crit}}$.

\begin{definition}
Let $P$ be a non-zero completely prime two-sided ideal and suppose that $b\in P$ is non-zero. Let
\[I_b^P:=\{c\in R \st b\notin cP\}\] and
\[{_b}{^P}I:=\{c\in R \st b\notin Pc\}.\]
\end{definition}

Note that $I_b^P$ is a right ideal and ${_b}{^P}I$ is a left ideal.

\begin{lemma}\label{Nminpair}
Let $N\in\Zg_R$. Suppose that $xb=0/a|x$ is an $N$-minimal pair such that $a,b\notin\ann_RN$. Then $N=N(I_b^P,{_a}{^Q}I)$ where $P:=\Ass N$ and $Q:=\Div N$.
\end{lemma}
\begin{proof}
Suppose that $xb=0> xc=0$ in $N$. Then there exists $m\in N$ such that $mb=0$ and $mc\neq 0$. Thus $b=cr$ and $r\in P$. So $c\notin I_b^P$.

 Suppose that $c\notin I_b^P$. Then $b=cr$ for some $r\in P$. Since $b\notin \ann_RN$, $c\notin \ann_RN$. Thus $m\in N_{\text{crit}}$ realising the critical type on $N$ is divisible by $c$ and $mr=0$. Take $n\in N_{\text{crit}}$ such that $m=nc$. Then $nc\neq 0$ and $ncr=0$. So $xb=0>xc=0$ in $N_{\text{crit}}$ and hence also in $N$. Thus $c\in I_b^P$ if and only if $xc=0\geq xb=0$.

A similar argument shows that $a|x\geq c|x$ if and only if $c\in {_a}{^Q}I$.

Suppose that $m\in N$ is such that $mb=0$ and $a$ does not divide $m$. Then by the above, $m$ has pp-type $\langle I_b^P,{_a}{^Q}I\rangle$ and thus $N=N(I_b^P,{_a}{^Q}I)$.
\end{proof}

\begin{proposition}\label{pointsuptotopind}
All points in $Zg_R$ are topologically indistinguishable from an indecomposable pure-injective of one of the following forms:
\begin{enumerate}
\item $N(P,K)$ where $K$ is a two-sided ideal, $P$ is a completely prime two-sided ideal and $rs\in K$ implies $r\in K$ or $s\in P$.
\item $N(K,P)$ where $K$ is a two-sided ideal, $P$ is a completely prime two-sided ideal and $sr\in K$ implies $r\in K$ or $s\in P$.
\item $N(I_b^P,{_a}{^Q}I)$ where $\langle I_b^P,{_a}{^Q}I\rangle$ is consistent.
\end{enumerate}
\end{proposition}
\begin{proof}
If $N$ is topologically indistinguishable from $N_{\text{crit}}$ or $N_{\text{gen}}$ then, by \ref{desccrit} and \ref{descgen}, $N$ is topologically indistinguishable from a module either from (1) or (2).

Suppose that $N$ is topologically distinguishable from $N_{\text{crit}}$ and $N_{\text{gen}}$. Then there exists a pair of pp-1-formulae such that $N\in\left(\phi/\psi\right)$ and $N_{\text{crit}},N_{\text{gen}}\notin \left(\phi/\psi\right)$. By \ref{uniserialringsgenandcrit}, we may assume that $\phi$ is $xb=0$ for some $b\in R$ and $\psi$ is $a|x$ for some $a\in R$. Moreover, $N_{\text{gen}}\notin \left(xb=0/a|x\right)$ implies $b\notin \ann_RN$ and $N_{\text{crit}}\notin\left(xb=0/a|x\right)$ implies $a\notin \ann_RN$.

Thus \ref{Nminpair} implies that $N\cong N(I_b^P,{_a}{^Q}I)$ where $P=\Ass N$ and $Q=\Div N$.

\end{proof}

Unfortunately, we are unable to find a simple characterisation of consistency of pairs of the form $\langle I_b^P,{_a}{^Q}I\rangle$. However, we will see in section \ref{Sexamples} that this is not always a big problem when computing Ziegler spectra of uniserial rings.

The following lemma shows that the critical and generic type of an indecomposable pure-injective almost determines it up to topological indistinguishability.

\begin{lemma}\label{critgenalmostdet}
Suppose $N,M\in\Zg_R$, $N_{\text{crit}}=M_{\text{crit}}$, $N_{\text{gen}}=M_{\text{gen}}$ and $M,N\notin \cl\{N_{\text{crit}},N_{\text{gen}}\}$. Then $N=M$.
\end{lemma}
\begin{proof}
Since $N$ and $M$ are not in the closure of $N_{\text{gen}}=M_{\text{gen}}$ or $N_{\text{crit}}=M_{\text{crit}}$, $N$ has a minimal pair $xb=0/a|x$ and $M$ has a minimal pair $xd=0/c|x$. Since $N_\text{crit}=M_{\text{crit}}$, $\Ass N=\Ass M=:P$. Thus $b,d\in P$. Since $N_{\text{gen}}=M_{\text{gen}}$, $\Div N=\Div M=:Q$. Thus $a,c\in Q$. Finally, $a,b\notin \ann_RN=\ann_RM$ and $c,d\notin \ann_R M=\ann_RN$.

Let $n\in N$ be such that $nb=0$ and $a$ does not divide $n$. So, by \ref{Nminpair},  $n$ has type $\langle I_b^P,{_a}{^Q}I\rangle$.

We may assume without loss of generality that $b=td$. Since $d\in P$, $t\notin I_b^P$ and hence $nt\neq 0$.

Thus $ntd=0$ and $at$ does not divide $nt$. Moreover $xd=0/at|x$ is an $N$-minimal pair since the interval $[at|x,xd=0]$ in $N$ is isomorphic to the interval $[a|x,xb=0]$ in $N$. Hence, $xd=0/at|x$  isolates $N$ in its closure.


Since $\ann_RN=\ann_RM$ and $xd=0>at|x$ in $N$, $atd\in\ann_R N=\ann_RM$ and therefore $xd=0\geq at|x$ in $M$.

Thus either $xd=0$ is equivalent to $at|x$ in $M$ and hence $N_{\text{gen}}=M_{\text{gen}}\in \left(at|x/c|x\right)$ and hence $N_{\text{gen}}=M_{\text{gen}}\in \left(xd=0/c|x\right)$ or $xd=0>at|x$ in $M$. The first possibility can't happen since $M_{\text{gen}}\notin \left(xd=0/c|x\right)$. Thus $xd=0>at|x$ in $M$. But if $xd=0/at|x$ is not a minimal pair for $M$ then $N_{\text{gen}}=M_{\text{gen}}\in \left(xd=0/at|x\right)$ or $N_{\text{crit}}=M_{\text{crit}}\in\left(xd=0/at|x\right)$. Both of which contradict our assumptions, thus $xd=0/at|x$ is a minimal pair for $M$ and $at|x$ is equivalent to $c|x$ in $M$. Thus, $at\notin \ann_RN=\ann_RM$.

Thus both $N$ and $M$ realise $\langle I_d^P,{_at}{^Q}I\rangle$ by \ref{Nminpair}. So $N=M$.

\end{proof}

\begin{lemma}\label{subbasis}
The sets $\left(x=x/xa=0\right)$, $\left(x=x/a|x\right)$, $\left(xa=0/x=0\right)$, $\left(xb=0/a|x\right)$ where $a,b\in R$ is a sub-basis for $\Zg_R$. Moreover,
\begin{enumerate}
\item $N\in \left(x=x/xa=0\right)$ if and only if $a\notin \ann_R N$
\item $N\in \left(xa=0/x=0\right)$ if and only if $a\in\Ass N$
\item $N\in\left(x=x/a|x\right)$ if and only if $a\in\Div N$
\end{enumerate}
\end{lemma}
\begin{proof}
By \cite[2.1]{Reynders}, the sets
\[\left(xa=0\wedge b|x/xc=0 +d|x\right)\] where $a,b,c,d\in R$ are a basis for $\Zg_R$. Since the lattice of pp-definable subsets of any indecomposable pure-injective module is totally ordered by inclusion,
\[\left(xa=0\wedge b|x/xc=0 +d|x\right)=\]\[\left(xa=0/xc=0\right)\cap\left(xa=0/d|x\right)\cap\left(b|x/xc=0\right)\cap\left(b|x/d|x\right).\]

If $\left(xa=0/xc=0\right)$ is non-empty then $a=cr$ for some $r\in R$ and
\[\left(xa=0/xc=0\right)=\left(xr=0/x=0\right)\cap \left(x=x/xc=0\right).\]
If $\left(b|x/d|x\right)$ is non-empty then $d=rb$ for some $r\in R$ and
\[\left(b|x/d|x\right)=\left(x=x/r|x\right)\cap\left(b|x/x=0\right)=\left(x=x/r|x\right)\cap\left(x=x/xb=0\right).\]
Finally
\[\left(b|x/xc=0\right)=\left(x=x/xbc=0\right).\]

\end{proof}

\begin{cor}\label{closure}
If $N_1$ is in the closure of $N_2$ then $\Ass N_2\supseteq \Ass N_1$, $\Div N_2\supseteq \Div N_1$ and $\ann_RN_1\supseteq \ann_RN_2$.
\end{cor}

\begin{question}
When is $\langle I_b^P,{_a}{^Q}I\rangle$ consistent? When is $\langle I_b^P,{_a}{^Q}I\rangle$ topologically distinguishable from both its generic and critical modules? What is the annihilator of $\langle I_b^P,{_a}{^Q}I\rangle$?
\end{question}

\section{Rank one uniserial domains}\label{Sexamples}

In this section we describe the Ziegler spectra, after factoring out by $T_0$, of all rank one uniserial domains $R$ according to the classification in \cite[1.9]{B-D}. A uniserial domain is rank one if its only non-zero completely prime ideal is $\J(R)$.

The classification in \cite[1.9]{B-D} is in terms of $P$-ideals where $P$ is a rank one cone of a group $G$ but we will make the straightforward translation into ideals of a rank one uniserial domains. Before giving the classification of rank one uniserial domains, we need to introduce the groupoid of divisorial ideals. We say that a right ideal $I$ is divisorial if
$I=\cap_{I\subseteq aR}aR=:\widehat{I}$. The ideal $\widehat{I}$ is called the divisorial closure of $I$. Clearly all principal right ideals are divisorial. Moreover, it is shown in \cite[pg 2737]{B-D} that a right ideal is not divisorial if and only if $\J(R)$ is not principal as a right ideal and there exists $z\in R$ such that $I=z\J(R)$ and $\widehat{I}=zR$. Moreover, in the case where $R$ is rank one and $I$ is an ideal, $I=z\J(R)=\J(R)z$ and $\widehat{I}=zR=Rz$.

The set of divisorial ideals becomes a groupoid under the multiplication $\star$ defined by $I_1\star I_2:=\widehat{I_1I_2}$.

According to \cite[1.9]{B-D}, every rank one uniserial domain $R$ has exactly one of the following properties.
\begin{enumerate}
\item $R$ is \textbf{invariant}, that is $aR=Ra$ for all $a\in R$.
\item $R$ is \textbf{nearly simple}, that is $R$ is not a division ring and $\J(R)$ is the only non-zero two-sided ideal.
\item $R$ is \textbf{exceptional}, that is, there exists a prime ideal $Q$ which is not completely prime. In this case,
\begin{enumerate}
\item there are no ideals between $\J(R)$ and $Q$,
\item the ideal $Q$ is divisorial and generates the groupoid of divisorial ideals,
\item $\cap_{n=1}^\infty Q^n=0$ and
\item there exists a $k\in \N_0$ such that $\widehat{Q^k}$ is principal and generates the groupoid of principal ideals. In this situation we say that $R$ is type $C_k$.
\end{enumerate}
\end{enumerate}

Importantly to us, the above classification gives the following descriptions of the chains of ideals.

If $R$ is of type $C_0$ then the chain of ideals is
\[R\supseteq \J(R)\supseteq Q\supseteq Q^2\supseteq \ldots.\] In this case $R$ has no ideals finitely generated as right ideals.

If $R$ is of type $C_1$ then the chain of ideals is
\[R\supseteq \J(R)\supseteq Q=zR\supseteq z\J(R)\supseteq Q^2=z^2R\supseteq z^2\J(R) \ldots.\]

If $R$ is of type $C_k$ for $k\geq 2$ then the chain of ideals is
\[R\supseteq \J(R)\supseteq Q \supseteq Q^2\supseteq \ldots\supseteq Q^{k-1} \supseteq zR\supseteq Q^k=z\J(R)\supseteq Q^{k+1}\ldots \]
\[\ldots Q^{2k-1}\supseteq z^2R\supseteq Q^{2k}=z^2\J(R)\supseteq Q^{2k+1}\ldots \]

Commutative valuation domains are examples of invariant rank one uniserial domains. Examples of nearly simple uniserial domains can be found in \cite[section 6.5]{rightchainrings} and \cite[3.8]{DubPun}. Finally, examples of exceptional uniserial rings of type $C_n$ for each $n\in\N$ are given in \cite{B-D}.

Before we go on, we sketch a proof of the following comforting fact, which we were unable to find a direct reference for.

\begin{lemma}\label{idealfgasleftideal}
If $R$ is a rank one uniserial domain, $a\neq 0$ and $aR$ is an ideal then $aR=Ra$.
\end{lemma}
\begin{proof}
If $I$ is an ideal then the following set is a completely prime ideal
\[\{b\in R \st bI\subsetneq I\}.\] It is straightforward to see that it is a right ideal and that it is completely prime. In order to show that it is a left ideal one needs to note, \cite{rightchainrings}, that a right ideal $K$ is a left ideal if and only if $uK\subseteq K$ for all $u\notin \J(R)$.

Thus setting $I=aR$ we have that $\{b\in R \st bI\subsetneq I\}=\J(R)$ since if it were the zero ideal then $a^2R=aR$ which would imply that $a(a\delta-1)=0$ for some $\delta\in R$ and hence $a$ would be a unit.

Now if $Ra\subsetneq aR$ then there exists $r\in R$ such that $ar\notin Ra$. So, since $R$ is uniserial there exists $\lambda\in \J(R)$ such that $\lambda ar=a$. But then $aR\subseteq \lambda aR$ and hence $\lambda\notin \J(R)$. So we have a contradiction.
\end{proof}

\subsection{Nearly simple uniserial domains}

We start by considering some rings with rather small Ziegler spectra.

\begin{remark}\label{finiteideallattfinitezg}
The results in section \ref{pointsforuniserial} imply that if a uniserial ring $R$ has only finitely many two-sided ideals then, after factoring out by $T_0$ its Ziegler spectrum is finite. This is because \ref{descgen} and \ref{desccrit} imply $\Zg_R$ has only finitely many critical and generic pairs, and \ref{critgenalmostdet} shows that if two indecomposable pure-injective modules $N,M$ have $N_{gen}=M_{gen}$, $N_{crit}=M_{crit}$ and both modules are topologically distinguishable from $N_{gen}$ and $N_{crit}$ then $N\cong M$.
\end{remark}

Since the Ziegler spectrum of a uniserial ring with only finitely many two-sided ideals is finite after factoring out by $T_0$, the topology is completely described by the specialisation relation on $\Zg_R/T_0$ ($x$ specialises to $y$ if $y\in\cl\{x\}$) and the $T_0$-equivalence classes.

The above remark in particular implies that all nearly simple uniserial rings have finite Ziegler spectra after factoring out by $T_0$, in fact we will show below that they are all homeomorphic to each other after factoring out by $T_0$. They are however very rarely finite before factoring out by $T_0$. At the end of this subsection, we will exhibit a rather surprising example, due to Dubrovin and Puninski, of a nearly simple uniserial domain with finite Ziegler spectrum even before factoring out by $T_0$.

\begin{lemma}
If $R$ is a nearly simple uniserial domain then all modules are topologically indistinguishable from critical or generic modules. The generic pairs are $\langle \J(R),\J(R) \rangle$, $\langle 0,\J(R)\rangle$ and $\langle 0,0\rangle$ and the critical pairs are $\langle \J(R),\J(R)\rangle$, $\langle \J(R),0\rangle$ and $\langle 0,0\rangle$.
\end{lemma}
\begin{proof}
Since $\J(R)$ is the unique non-zero two-sided ideal of $R$, by \ref{descgen}, the generic pairs are $\langle \J(R),\J(R) \rangle$, $\langle 0,\J(R)\rangle$ and $\langle 0,0\rangle$ and, by \ref{desccrit}, the critical pairs are $\langle \J(R),\J(R)\rangle$, $\langle \J(R),0\rangle$ and $\langle 0,0\rangle$. If $N$ is a module with $\ann_RN=0$ then $N\in \left(xb=0/a|x\right)$ implies $ab=0$. Thus $a=0$ or $b=0$. Considering the sub-basis from \ref{subbasis}, implies that if $\ann_RN=0$ then $N$ is topologically indistinguishable from its critical module or its generic module. On the other hand, if $\ann_RN=\J(R)$, then $\Ass N=\J(R)$ and $\Div N=\J(R)$. Thus $N=N(\J(R),\J(R))$. Therefore, all points are topologically indistinguishable from either generic or critical modules.
\end{proof}

The points of the Ziegler spectrum were described as modules in \cite{Punnearsim}, we have added descriptions of their critical and generic types. Here, $\PE(M)$ denotes the pure-injective hull of $M$ and $E(M)$ denotes the injective hull of $M$.

\begin{proposition}\label{indpinearsim}
Let $R$ be a nearly simple uniserial domain. The indecomposable pure-injective right modules are the following:
\begin{enumerate}
\item indecomposable injective modules $E(R/I)$ where $I$ is a non-zero right ideal. This module corresponds to the consistent pair $\langle I,0\rangle$ and $E(R/I)\cong E(R/J)$ if and only if there exists $r\in R$ such that $rI=J$ or $I=rJ$. The critical pair of $E(R/I)$ is $\langle \J(R),0\rangle$ and the generic pair is $\langle 0,0\rangle$.
\item indecomposable pure-injective torsion-free modules $\PE(I)$ where $I$ is a non-zero right ideal. This module corresponds to the consistent pair $\langle 0,I\rangle$ and $\PE(I)\cong\PE(J)$ if and only if there exists $r\in R$ such that $rI=J$ or $I=rJ$. The generic pair of $\PE(I)$ is $\langle 0,\J(R)\rangle$ and the critical pair is $\langle 0,0\rangle$.
\item $R/\J(R)$ This module corresponds to the consistent pair $\langle \J(R),\J(R)\rangle$ which is both critical and generic.
\item $E(R_R)$ This module corresponds to the consistent pair $\langle 0,0\rangle$ which is both critical and generic.
\end{enumerate}

\end{proposition}

\begin{lemma}\label{topnearsim}
Let $R$ be a nearly simple uniserial domain. After identifying topologically indistinguishable points
\begin{enumerate}
\item $\cl N(\J(R),0)=\{N(0,0), N(\J(R),0)\}$
\item $\cl N(0,\J(R))=\{N(0,0), N(0,\J(R))\}$
\item $\cl N(\J(R),\J(R)) = \{N(\J(R),\J(R))\}$
\item $\cl N(0,0)=\{N(0,0)\}$
\end{enumerate}
\end{lemma}
\begin{proof}
For each of the above equalities the inclusion of the left hand set in the right is implied by \ref{closure}. Since $\langle 0,0\rangle$ is the generic pair of $\langle \J(R),0\rangle$ (respectively critical pair of $\langle 0,\J(R)\rangle$), $N(0,0)$ is in the closure of $N(\J(R),0)$ (respectively $N(0,0)$ is in the closure of $N(0,\J(R))$).
\end{proof}

We will now give an example, taken from \cite[3.8]{DubPun}, of a nearly simple uniserial domain whose Ziegler spectrum is finite even before factoring out by $T_0$.

Let $G$ be the group of affine linear transformations of the real plane, whose elements are linear increasing functions $f= at+ b$, where
$0< a$ and $b$ are reals, and the multiplication is the composition:
\[(at+b)\cdot (ct+d)= a(ct+d)+ b= act+ (ad+b).\] The identity function $t$ is the unit of $G$
and $(at+b)^{-1}= a^{-1} t- a^{-1}b$. The group ring $FG$ over a field $F$ is a left and right Ore domain. This is because $G$ is a semi-direct product of the normal subgroup $N:=\{t+b\st b\in \R\}$ and the subgroup $L:=\{at \st 0<a\in\R\}$ and both these groups are torsion-free and abelian. See also \cite[3.2]{PrihPun} where $\R$ replaced is with $\Q$.


Fix an irrational $\epsilon$ and consider the set $P$ of the functions $f\in G$ such that $f(\epsilon)\geq \epsilon$. Then
$P$ is a \textbf{right cone} on $G$, i.e. $P\cup P^{-1}= G$; further $P\cap P^{-1}$ consists of the functions
$f\in G$ such that $f(\epsilon)= \epsilon$. It follows that the relation $g\leq_l h$ if $g^{-1}h\in P$ defines a
\textbf{left linear quasi-ordering} on $G$, where $g\leq_l h$ if and only if $g(\epsilon)\leq h(\epsilon)$; and this ordering respects
left multiplication by elements of $G$.

Similarly, setting $f\leq_r g$ if $gf^{-1}\in P$ we obtain the \textbf{right linear quasi-ordering} on $G$,
which respects right multiplication by elements of $G$. Note that $f\leq_r g$ if and only if the intercept of $f$ with
the vertical line $y= \epsilon$ is to the right of the corresponding intercept for $g$.

Let $P^+$ be the subsemigroup of $P$ consisting of the functions $f\in P$ such that $f(\epsilon)> \epsilon$. Then $T= FP\setminus FP^+$ is a left and right Ore set in the semigroup ring $FP$. Since $FP$ is a domain and it is fairly clear that $T$ is a multiplicatively closed subset. So we just need to show that $T$, in the terminology of \cite{Lam}, it is left and right permutable. To show this is slightly more complicated than \cite[3.3]{PrihPun} because there exist $g\in G$ such that $g(\epsilon)=\epsilon$ where $g$ is not the identity function. Despite this, the proof still works if instead of rewriting terms in $FG$ as $(1+\sum\alpha_ih_i)\alpha u$ where $u\in G$, $\alpha,\alpha_1,\ldots,\alpha_n\in F$ and all $h_i\in P^+$ we rewrite them as $(\sum_{i=1}^n\alpha_ig_i+\sum_{i=1}^m\beta_ih_i)\alpha u$ where $u\in G$, $\alpha, \alpha_1,\ldots,\alpha_n,\beta_1,\ldots,\beta_m\in F$, all $h_i\in P^+$ and $g_i(\epsilon)=\epsilon$ for all $1\leq i\leq n$.

Finally the localisation $R:=(FP) T^{-1}=T^{-1}(FP)$ is a nearly simple uniserial domain.

Each $r\in R$ can be written as a fraction $(\sum \alpha_i f_i)\cdot u^{-1}$, where $f_i\in P$ and
$u\in FP\backslash FP^+$. By ordering the $f_i$ with respect to $\leq_l$ we see that $rR$ equals $f_iR$ for some $i$. Further,
$f\in gR$ for $f, g\in P$ if and only if $f(\epsilon)\geq g(\epsilon)$, hence the principal right ideals are linearly ordered according
to the values of their generators at $\epsilon$. For instance, $f\in P$ is invertible in $R$ if and only if $f(\epsilon)= \epsilon$.

Similarly each $r\in R$ can be written as a left fraction $u^{-1} \sum \delta_l v_l$ for some
$u\in FP\backslash FP^+$ and $v_l\in P$, hence the left ideal $Rr$ equals $Rv_l$, where $v_l$ is the least element in the
$\leq_r$ ordering. Further, $f\in Rg$ for $f, g\in P$ if and only if $f\geq_r g$, i.e. the principal left ideals of $R$
are linearly ordered by the intercept values of their generators.

That $R$ is nearly simple follows exactly as in \cite[pg52-53]{rightchainrings}.


For our purposes, the most important property of the nearly simple uniserial ring $R$ is that all right ideals of $R$ are of the form $aR$ or $a\J(R)$ and all left ideals of $R$ are of the form $Ra$ or $\J(R)a$. This is because right ideals in $R$ correspond to cuts in the chain of principal right ideals of $R$ and the chain of principal right ideals of $R$ is isomorphic to the reals greater than $\epsilon$ as an order via the map $fR$ for $f\in P$ maps to $f(\epsilon)$. A similar justification for left ideals holds with the value of $f(\epsilon)$ replaced by the value of $x$ such that $f(x)=\epsilon$.

Proposition \ref{indpinearsim} implies that if $0\neq r,s\in R$ then $N(rR,0)\cong N( sR,0)$, $N(r\J(R),0)\cong N(\J(R),0)$, $N(0,Rr)\cong N(0,Rs)$ and $N(0,\J(R)r)\cong N(0,\J(R))$. Thus the following lemma follows from \ref{indpinearsim}.

\begin{lemma}\label{near-zig}(see \cite{Punnearsim})
Let $R$ be the above nearly simple uniserial domain and choose $0\neq r\in \J(R)$. The Ziegler spectrum of $R$
consists of the following 6 points, where we include for each point a corresponding consistent pair.
\begin{enumerate}
\item The simple module $R/\J(R)$ corresponds to the pair $\langle\J(R),\J(R)\rangle$.
\item The injective hull $E(R/\J(R))$ corresponds to the pair $\langle 0,\J(R)\rangle$.
\item The injective hull $E(R/rR)$ corresponds to the pair $\langle 0,rR\rangle$.
\item The pure-injective hull $\PE(R_R)$ corresponds to the pair $\langle\J(R),0\rangle$.
\item The pure-injective hull $\PE(\J(R))$ corresponds to the pair $\langle Rr,0\rangle$.
\item The division ring of fractions $Q(R)$ corresponds to the pair $\langle 0,0\rangle$.
\end{enumerate}
Moreover, $E(R/\J(R))$ and $E(R/rR)$ are topologically indistinguishable from one another, and  $\PE(R_R)$ and $\PE(\J(R))$ are topologically indistinguishable from one another.
\end{lemma}

The topology in now described by \ref{topnearsim}.

\subsection{Invariant rank one uniserial domains}

The case of invariant rank one uniserial domains will turn out to be exactly as for commutative rank one uniserial domains i.e. valuation domains, in fact from what we show it will follow that if $R$ is an invariant rank one uniserial domain then there exists a valuation domain $S$ with $\Zg_R$ homeomorphic to $\Zg_S$.

\begin{lemma}
Let $R$ be an invariant uniserial domain. Then all pairs of proper ideals of $R$ are consistent.
\end{lemma}
\begin{proof}
Let $I,J\lhd R$. Suppose that $r\in I,s\in J,r^*\notin I$ and $s^*\notin J$. Since $I$ and $J$ are ideals, there exists $\lambda,\mu\in R$ such that $s=s^*\lambda$ and $r=\mu r^*$. Suppose for a contradiction that $s^*r^*\in RsrR=srR$. Then $r^*\in \lambda rR$ since $R$ is a domain. Again using that $R$ is a domain, $r^*\in \lambda rR=R\lambda r$ implies $1\in R\lambda\mu$ and hence $\mu$ and $\lambda$ are units. But this implies $r^*\in I$ and $s^*\in J$. So we have a contradiction and hence $s^*r^*\notin RsrR=srR$. So by \ref{Genasconsistencycondition}, $\langle I,J\rangle$ is consistent.
\end{proof}

Let $R$ be a uniserial domain with group of units $U(R)$. Then $R$ is Ore. Let $Q$ be the division ring of fractions of $R$. Note that for all $q\in Q^{\times}$, either $q\in R$ or $q^{-1}\in R$ and in fact, any subring of a division ring with this property is a uniserial domain. From now on assume $R$ is invariant. This implies that $U(R)$ is a normal subgroup of $Q^{\times}$. As in the (commutative) valuation domain case, we will call the group $\Gamma_R:=Q^{\times}/U(R)$ the value group of $R$ and let $v:R\rightarrow \Gamma_R\cup\{\infty\}$ be the map which send $0\neq r\in R$ to $rU(R)$ and $0$ to $\infty$. Note that this group has a total order given by $b\leq a$ if $ab^{-1}\in R$, equivalently if $b^{-1}a\in R$.

Ideals in $R$ correspond to upsets in $\Gamma_R\cup\{\infty\}$ via the bijection
\[I\lhd R \mapsto v(I)\]

Now $R$ is rank one if and only if $\Gamma_R$ is archimedean, i.e. for all $0<a,b\in \Gamma_R$ there exists $n\in\N$ such that $b<a^n$. By a theorem of H\"{o}lder, see for instance \cite[2.2.1]{rightorderedgroups}, all archimedean totally order groups are subgroups of $\R$. So in particular $\Gamma_R$ is abelian.

In \cite{Lornasober} the Ziegler spectrum of a valuation domain was described in terms of its value group. We now recall this description and note that it also works for rank one invariant uniserial domains.

If $E$ is a strictly positive upset in $\Gamma\cup\{\infty\}$ and $\gamma\in \Gamma_{\geq 0}$ then let
\[E+\gamma:=\{e+\gamma \st e\in E\}.\] Note that $E+\gamma$ is a strictly positive upset and moreover, if $I$ is an ideal in $R$ such that $v(I)=E$ and $g\in R$ is such that $v(g)=\gamma$ then $v(Ig)=v(gI)=E+\gamma$.

We define an equivalence relation $\approx$ on the set of pairs of strictly positive upsets of $\Gamma_{\geq 0}\cup\{\infty\}$.

If $\langle E,F\rangle, \langle G,H\rangle$ are pairs of strictly positive upsets then $\langle E,F\rangle \approx\langle G,H\rangle$ if one of the following is true
\begin{enumerate}
\item there exists $\gamma\in \Gamma_{\geq 0}$ such that $E=G+\gamma$ and $F+\gamma=H$
\item there exists $\gamma\in \Gamma_{\geq 0}$ such that $E+\gamma=G$ and $F=H+\gamma$.
\end{enumerate}

Note that by \ref{equonepairs}, pairs of ideals $\langle I,J\rangle$ and $\langle K,L\rangle$ are such that $N(I,J)\cong N(K,L)$ if and only if $\langle v(I),v(J)\rangle\approx \langle v(K),v(L)\rangle$.

Write $(E,F)$ for the $\approx$-equivalence class of $\langle E,F\rangle$ and write $Z_R$ for the set of $\approx$-equivalence classes.

For $\alpha,\beta\in \Gamma_{\geq 0}$, $\gamma,\delta\in\Gamma_{>0}\cup\{\infty\}$, let $\mcal{W}_{\alpha,\beta,\gamma,\delta}$ be the set of $\approx$-equivalence classes of pairs of strictly positive upsets $(E,F)$ such that there exists a pair $\langle G,H\rangle$ in the same $\approx$-equivalence class as $\langle E,F\rangle$ with $\alpha\notin G, \beta\notin H,\alpha+\gamma\in G \text{ and } \beta+\delta\in H$.

According to \cite[2.1]{Reynders}, the sets
\[\left(xa=0\wedge b|x/xc=0 +d|x\right)\]
where $a,b,c,d\in R$ are a basis for $\Zg_R$.

One can show, just as in the commutative valuation domain case \cite[2.4]{CBrankVdom}, that such a set is non-empty if and only if $c\neq 0$, $a=cr$ for some $r\in \J(R)$, $b\neq 0$ and $d=sb$ for some $s\in \J(R)$. Moreover $N(I,J)\in \left(xcr=0\wedge b|x/xc=0 +sb|x\right)$ if and only if there exists some consistent pair of  ideal $(K,L)$ such that $N(K,L)\cong N(I,J)$ and $c\notin K$, $b\notin L$, $cr\in K$ and $sb\in L$.

Thus $N(I,J)\in \left(xcr=0\wedge b|x/xc=0 +sb|x\right)$ if and only if $(v(I),v(J))\in \mcal{W}_{v(c),v(b),v(r),v(s)}$.

So exactly as in \cite[3.2]{Lornasober} the set $Z_R$ equipped with a topology by taking the sets $\mcal{W}_{\alpha,\beta,\gamma,\delta}$ as a basis of open sets is homeomorphic to $\Zg_R$.

Thus if $R$ is a rank one invariant uniserial domain then, by \cite[3.8]{MONND}, there exists a valuation domain $S$ with value group $\Gamma_R$ and thus $\Zg_R$ is homeomorphic to $\Zg_S$.

We now give an explicit description of $\Zg_R$ in terms of $\Gamma_R$. All archimedean totally ordered abelian groups are either isomorphic to $\Z$ or are a dense subgroup of $\R$.

In the following two computations we write $\Gamma$ instead of $\Gamma_R$ to simplify notation.

\medskip

\noindent
\underline{$\Gamma\cong \Z$}

In this case $\Zg_R$ is homeomorphic to the Ziegler spectrum of a discrete valuation domain. The Ziegler spectra of discrete valuation domains are very well known, see for instance \cite[section 5.2.1]{PSL}.

Note that all upsets in $\Gamma_{>0}\cup\{\infty\}$ are either of the form $\Gamma_{\geq n}:=\{m\in \Z \st m\geq n\}$ or $\{\infty\}$.

For each $n\in\N$, there is a point, which we will label $n$, corresponding to the pair $\langle \Gamma_{\geq 1},\Gamma_{\geq n}\rangle $, this corresponds to the module $R/J(R)^n$. The remaining points are $ \infty$ corresponding to the equivalence class of the pair $\langle\{\infty\},\{\infty\} \rangle$, $\infty^+ $ corresponding to the equivalence class of the pair $\langle\{\infty\},\Gamma_{\geq 1}\rangle$ and $\infty^- $ corresponding the equivalence class of the pair $\langle\Gamma_{\geq 1},\{\infty\}\rangle$. This is a complete list of $\approx$-equivalence classes.

Either by direct computation or by comparing with the description of the topology for a discrete valuation domain given in \cite{PSL}, we have the following description of the topology.

A subset $C \subseteq\Zg_R$ is closed if and only if the following two properties hold
\begin{enumerate}
\item If $\infty^+\in C$ or $\infty^-\in C$ then $\infty\in C$.
\item If $n\in C$ for infinitely many $n\in\N$ then $\infty,\infty^+,\infty^{-}\in C$.
\end{enumerate}

\medskip

\noindent
\underline{$\Gamma$ is a dense subgroup of $\R$}

Strictly positive upsets in $\Gamma\cup \{\infty\}$ all have one of the following forms
\begin{enumerate}[(i)]
\item $I_{>a}:=\{b\in \Gamma_{\geq 0} \st b>a\}$ where $a\in \R_{\geq 0}$
\item $I_{\geq a}:=\{b\in \Gamma_{\geq 0} \st b\geq a\}$ where $a\in \Gamma_{>0}$
\item $\{\infty\}$
\end{enumerate}

One can check that the following is a lists of representatives of $\approx$-equivalence classes (note that this list contains many more that one representative for each class).
\begin{enumerate}[(i)]
\item $\langle \{\infty\}, I_{>a}\rangle$, $\langle \{\infty\},I_{\geq b}\rangle$ for $a\in \R_{\geq 0}$ and $b\in\Gamma_{>0}$
\item   $\langle I_{>a},\{\infty\}\rangle$, $\langle I_{\geq b}, \{\infty\}\rangle$ for $a\in \R_{\geq 0}$ and $b\in\Gamma_{>0}$
\item $\langle\infty,\infty\rangle$
\item $\langle I_{>a},I_{>b}\rangle$ for $a,b\in \R_{\geq 0}$
\item $ \langle I_{\geq a},I_{\geq b}\rangle$ for $a,b\in \Gamma_{>0}$
\item $\langle I_{> a},I_{\geq b}\rangle$ for $a\in \R_{\geq 0}$ and $b\in \Gamma_{>0}$
\item $\langle I_{\geq a},I_{> b}\rangle$ for $b\in \R_{\geq 0}$ and $a\in \Gamma_{>0}$
\end{enumerate}

We will now consider the space $Z_R$ after identifying topologically indistinguishable points.

Let $\alpha,\beta\in\Gamma_{\geq 0}$ and $\gamma,\delta\in\Gamma_{>0}\cup\{\infty\}$. One can show that each of the points listed under (i) are in $\mcal{W}_{\alpha,\beta,\gamma,\delta}$ if and only if $\gamma=\infty$. Label this point as $\infty^+$. Symmetrically, each of the points listed under (ii) are in $\mcal{W}_{\alpha,\beta,\gamma,\delta}$ if and only if $\delta=\infty$. Label this point as $\infty^-$. The point $(\{\infty\},\{\infty\})$ is in $\mcal{W}_{\alpha,\beta,\gamma,\delta}$ if and only if $\gamma=\delta=\infty$. Label this point $\infty$.

We first state some facts about dense subgroups of $\R$ which will allow us to describe the topology on $Z_R$.

Let $\alpha,\beta\in\Gamma_{\geq 0}$, $\gamma,\delta\in\Gamma_{>0}\cup\{\infty\}$ and $\epsilon_1,\epsilon_2\in \Gamma$.

\begin{enumerate}
\item There exists $\mu\in\Gamma$ such that
\[\alpha\leq \epsilon_1+\mu <\alpha +\gamma \text{ and } \beta \leq \epsilon_2 -\mu <\beta + \delta\]
if and only if
\[\alpha+\beta\leq \epsilon_1+\epsilon_2 <(\alpha +\gamma)+(\beta +\delta).\]
\item There exists $\mu\in\Gamma$ such that
\[\alpha< \epsilon_1+\mu \leq\alpha +\gamma \text{ and } \beta< \epsilon -\mu \leq\beta + \delta\]
if and only if
\[\alpha+\beta< \epsilon_1+\epsilon_2 \leq(\alpha +\gamma)+(\beta +\delta).\]
\item There exists $\mu\in\Gamma$ such that
\[\alpha\leq \epsilon_1+\mu <\alpha +\gamma \text{ and }\beta< \epsilon_2 -\mu \leq\beta + \delta\]
if and only if
\[\alpha+\beta< \epsilon_1+\epsilon_2 <(\alpha +\gamma)+(\beta +\delta).\]
\end{enumerate}



This shows that for $\alpha,\beta\in\Gamma_{\geq 0}$, $\gamma,\delta\in\Gamma_{>0}\cup\{\infty\}$ and $\epsilon_1,\epsilon_2\in \Gamma$
\begin{enumerate}
\item $(I_{>\epsilon_1},I_{>\epsilon_2})\in \mcal{W}_{\alpha,\beta,\gamma,\delta}$ if and only if \[\alpha+\beta\leq \epsilon_1+\epsilon_2 <(\alpha +\gamma)+(\beta +\delta).\] We label this point $(\epsilon_1+\epsilon_2)^+$.
\item $(I_{\geq\epsilon_1},I_{\geq\epsilon_2})\in \mcal{W}_{\alpha,\beta,\gamma,\delta}$ if and only if \[\alpha+\beta< \epsilon_1+\epsilon_2 \leq(\alpha +\gamma)+(\beta +\delta).\] We label this point $(\epsilon_1+\epsilon_2)^-$.
\item $(I_{>\epsilon_1},I_{\geq\epsilon_2})\in \mcal{W}_{\alpha,\beta,\gamma,\delta}$ (respectively $(I_{\geq\epsilon_1},I_{>\epsilon_2})\in \mcal{W}_{\alpha,\beta,\gamma,\delta}$) if and only if
    \[\alpha+\beta< \epsilon_1+\epsilon_2 <(\alpha +\gamma)+(\beta +\delta).\] We label both these points $(\epsilon_1+\epsilon_2)$.
\end{enumerate}

We now consider pairs involving $I_{>\epsilon}$ where $\epsilon\notin\Gamma$.

Let $\alpha,\beta\in\Gamma_{\geq 0}$, $\gamma,\delta\in\Gamma_{>0}\cup\{\infty\}$, $q\in\Gamma$ and $\epsilon_1,\epsilon_2\in \R\backslash\Gamma$.
\begin{enumerate}
\item There exists $\mu\in\Gamma$ such that
\[\alpha\leq \epsilon_1+\mu<\alpha+\gamma \text{ and } \beta\leq \epsilon_2-\mu<\beta+\delta\]
if and only if
\[\alpha+\beta<\epsilon_1+\epsilon_2<(\alpha+\gamma)+(\beta+\delta)\]
\item There exists $\mu\in\Gamma$ such that
\[\alpha\leq \epsilon_1+\mu<\alpha+\gamma \text{ and }\beta\leq q-\mu<\beta+\delta\]
if and only if
\[\alpha+\beta<\epsilon_1+q<\alpha+\gamma+\beta+\delta\]
\item There exists $\mu\in\Gamma$ such that
\[\alpha\leq \epsilon_1+\mu<\alpha+\gamma \text{ and }\beta< q-\mu\leq\beta+\delta\]
if and only if
\[\alpha+\beta<\epsilon_1+q<(\alpha+\gamma)+(\beta+\delta)\]
\end{enumerate}



This shows that for $\alpha,\beta\in\Gamma_{\geq 0}$, $\gamma,\delta\in\Gamma_{>0}\cup\{\infty\}$, $q\in\Gamma$ and $\epsilon_1,\epsilon_2\in \R\backslash\Gamma$
\begin{enumerate}
\item $(I_{>\epsilon_1},I_{>\epsilon_2})\in \mcal{W}_{\alpha,\beta,\gamma,\delta}$ if and only if
\[\alpha+\beta<\epsilon_1+\epsilon_2<(\alpha+\gamma)+(\beta+\delta).\] We label this point $(\epsilon_1+\epsilon_2)$.
\item $(I_{>\epsilon_1},I_{>q}) \in \mcal{W}_{\alpha,\beta,\gamma,\delta}$ if and only if
$(I_{>q},I_{>\epsilon_1}) \in \mcal{W}_{\alpha,\beta,\gamma,\delta}$
\[\alpha+\beta<\epsilon_1+q<\alpha+\gamma+\beta+\delta.\] We label both these points $(\epsilon_1+q)$.
\item $(I_{>\epsilon_1},I_{\geq q}) \in \mcal{W}_{\alpha,\beta,\gamma,\delta}$ if and only if
$(I_{\geq q},I_{>\epsilon_1}) \in \mcal{W}_{\alpha,\beta,\gamma,\delta}$ if and only if
\[\alpha+\beta<\epsilon_1+q<(\alpha+\gamma)+(\beta+\delta)\] We label both these points $(\epsilon_1+q)$.
\end{enumerate}

We have thus now described the topology on $Z_R$. To summarise, if $R$ is an invariant uniserial domain with value group $\Gamma$ which is a dense subgroup of $\R$ then $\Zg_R/T_0$ is homeomorphic to a topological space with set of points
\[\{r \st r\in \R_{>0}\}\cup\{r^+\st r\in \Gamma_{\geq}0\}\cup\{r^-\st r\in\Gamma_{>0}\}\cup\{\infty,\infty^+,\infty^-\}\] and basis of open sets $\mcal{W}_{\alpha,\beta,\gamma,\delta}$ where $\alpha,\beta\in\Gamma_{\geq 0}$ and $\gamma,\delta\in\Gamma_{>0}\cup\{\infty\}$ such that
\begin{enumerate}
\item $\infty\in \mcal{W}_{\alpha,\beta,\gamma,\delta}$ if and only if $\gamma=\delta=\infty$,
\item $\infty^+\in \mcal{W}_{\alpha,\beta,\gamma,\delta}$ if and only if $\gamma=\infty$,
\item $\infty^-\in \mcal{W}_{\alpha,\beta,\gamma,\delta}$ if and only if $\delta=\infty$,
\item $r\in \mcal{W}_{\alpha,\beta,\gamma,\delta}$ if and only if $\alpha+\beta<r<\alpha+\gamma+\beta+\delta$,
\item $r^+\in \mcal{W}_{\alpha,\beta,\gamma,\delta}$ if and only if $\alpha+\beta\leq r<\alpha+\gamma+\beta+\delta$ and
\item $r^-\in \mcal{W}_{\alpha,\beta,\gamma,\delta}$ if and only if $\alpha+\beta<r \leq\alpha+\gamma+\beta+\delta$.
\end{enumerate}

%

The following picture shows a typical open set in $Z_R$ with $\infty^-,\infty,\infty^+$ removed. The oscillating lines show specialisation between points i.e. $\cl r^-=\{r,r^-\}$, $\cl r^+=\{r,r^+\}$ and $r$ is a closed point.

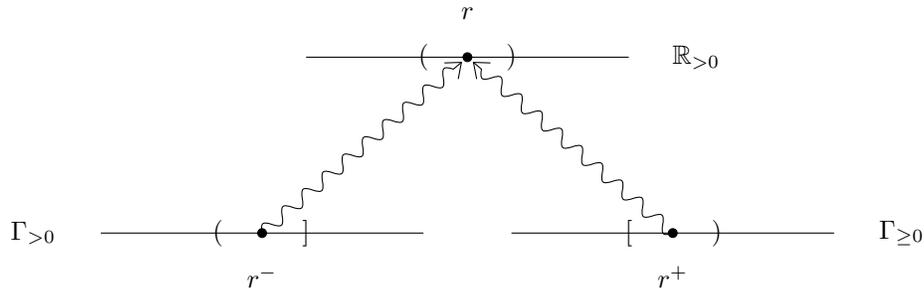
{\centering
\begin{tikzpicture}[scale=0.39]
\draw (2,4) --(13,4);
\draw (16,4) --(27,4);
\draw (9,10) --(20,10);

\draw[fill] (7.5,4) circle [radius=0.15];
\node at (7.5,2.5) {$r^{-}$};
\node at (-0.3, 4) {$\Gamma_{>0}$}; 
\node at (6,4) {$($};
\node at (9,4) {$]$};
\draw[fill] (21.5,4) circle [radius=0.15];
\node at (21.5,2.5) {$r^{+}$};
\node at (29.3,4) {$\Gamma_{\geq 0}$};
\node at (20,4) {$[$};
\node at (23,4) {$)$};
\draw[fill] (14.5,10) circle [radius=0.15];
\node at (14.5,11.5) {$r$};
\node at (22.3, 10) {$\R_{>0}$};
\node at (13,10) {$($};
\node at (16,10) {$)$};
\draw[-biggertip,shorten >=3pt,snake=coil,segment aspect=0] (7.5,4) -- (14.5,10);
\draw[-biggertip,shorten >=3pt,snake=coil,segment aspect=0] (21.5,4) -- (14.5,10);

\end{tikzpicture}
}

\subsection{Exceptional rank one uniserial domains}

In what remains of this paper we will prove the following three theorems which describe the Ziegler spectra of exceptional rank one uniserial domains after factoring out by $T_0$.

\begin{theorem}\label{typeC0}
Let $R$ be an exceptional rank one uniserial domain of type $C_0$. Up to $T_0$ the points of $\Zg_R$ are
\begin{enumerate}
\item $N(0,0)$, $N(\J(R),\J(R))$, $N(\J(R),0)$, $N(0,\J(R))$ and
\item for each $n\in\N$, $X_n:=N(\J(R),Q^n)\equiv N(Q^n,\J(R))$.
\end{enumerate}

A subset $C$ of $\Zg_R/T_0$ is closed if and only if the following conditions are satisfied:
\begin{enumerate}[(a)]
\item If $N(\J(R),0)\in C$ or $N(0,\J(R))\in C$ then $N(0,0)\in C$.
\item If $C$ contains infinitely many $X_n$ then $N(\J(R),0),N(0,\J(R))$ and $N(0,0)$ are all in $C$.
\end{enumerate}

\end{theorem}

\begin{theorem}\label{typeC1}
Let $R$ be an exceptional rank one uniserial domain of type $C_1$.
Up to $T_0$ the points of $\Zg_R$ are
\begin{enumerate}
\item $N(0,0)$, $N(\J(R),\J(R))$, $N(\J(R),0)$, $N(0,\J(R))$,
\item for each $n\in\N$, $X_n:=N(z^n\J(R),\J(R))\equiv N(\J(R),\J(R)z^n)$,
\item for each $n\in\N$, $Z_n:=N(z^nR,\J(R))\equiv N(\J(R),Rz^n)$ and
\item for each $n\in\N$, $Y_{n+1}:=N(zR,Rz^n)$ and $Y_1=N(bR,Ra)$ where $ab=z$.
\end{enumerate}
A subset $C$ of $\Zg_R/T_0$ is closed if and only if the following conditions are satisfied:
\begin{enumerate}[(a)]
\item If $X_n\in C$ then $Z_n\in C$.
\item If $Y_n\in C$ then $Z_n\in C$.
\item If $N(\J(R),0)\in C$ or $N(0,\J(R))\in C$ then $N(0,0)\in C$.
\item If $C$ contains infinitely many $X_n$, $Y_n$ or $Z_n$ then $N(\J(R),0)$, $N(0,\J(R))$ and $N(0,0)$ are all in $C$.
\end{enumerate}

\end{theorem}

\begin{theorem}\label{typeCk}
Let $R$ be an exceptional rank one uniserial domain of type $C_k$ for $k>1$.
Up to $T_0$ the points of $\Zg_R$ are
\begin{enumerate}
\item $(0,0)$, $N(\J(R),\J(R))$, $N(\J(R),0)$ and $N(0,\J(R))$,
\item for each $n\in\N$, $X_n:=N(\J(R),Q^n)\equiv N(Q^n,\J(R))$,
\item for each $n\in\N$, $Y_{n+1}:=N(zR,Rz^n)$ and $Y_1=N(bR,Ra)$ where $ab=z$ and
\item for each $n\in\N$, $Z_n:=N(\J(R),Rz^n)\equiv N(z^nR,\J(R))$.
\end{enumerate}

A subset $C$ of $\Zg_R/T_0$ is closed if and only if the following conditions are satisfied:
\begin{enumerate}[(a)]
\item If $Y_n\in C$ then $Z_n\in C$.
\item If $X_{kn}\in C$ then $Z_n\in C$.
\item If $N(\J(R),0)\in C$ or $N(0,\J(R))\in C$ then $N(0,0)\in C$.
\item If $C$ contains infinitely many $X_n$, $Y_n$ or $Z_n$ then $N(\J(R),0)$, $N(0,\J(R))$ and $N(0,0)$ are all in $C$.
\end{enumerate}

\end{theorem}

In the above theorems, all but the final condition describing closed sets describe the specialisation relation.

We start by proving some general results. The following lemma describes all critical and generic modules.

\begin{lemma}
Let $R$ be a rank one uniserial domain. If $K$ is a two-sided ideal then $\langle\J(R),K\rangle$ and $\langle K,\J(R)\rangle$ are consistent. Moreover, if $K$ is a non-zero two-sided ideal then the critical type of $N(K,\J(R))$ is $\langle\J(R),K\rangle$ and the generic type of $N(\J(R),K)$ is $\langle K,\J(R)\rangle$. The critical type of $N(0,\J(R))$ is $\langle 0,0\rangle$ and the generic type of $N(\J(R),0)$ is $\langle 0,0\rangle$.
\end{lemma}
\begin{proof}
Suppose $rs\in K$ and $s\notin \J(R)$. Then $s$ is a unit. Since $K$ is two-sided, $r=rss^{-1}\in K$. So, by \ref{desccrit}, $\langle\J(R),K\rangle$ is consistent. That $\langle K,\J(R)\rangle$ is consistent follows in the same way, this time using \ref{descgen}.

Suppose that $K$ is non-zero. Since $\langle K,\J(R)\rangle$ is a generic type, $\Div N(K,\J(R))=\J(R)$ and $\ann_R N(K,\J(R))=K$. Since $\ann_RN(K,\J(R))$ is non-zero and $\J(R)$ is the only non-zero completely prime two-sided ideal, $\Ass N(K,\J(R))=\J(R)$. Thus the critical type of $N(K,\J(R))$ is $\langle\J(R),K\rangle$. Symmetrically, the generic type of $N(\J(R),K)$ is $\langle K,\J(R)\rangle$.

Since $\langle 0,\J(R)\rangle$ is a generic type, $\Div N(0,\J(R))=\J(R)$ and $\ann_R N(0,\J(R))=0$. If $\langle I,J\rangle$ realised in $N(0,\J(R))$ then $aI=0$ for some non-zero $a\in R$. Thus $I=0$. Thus $\Ass N(0,\J(R))=0$. So the critical type of $N(0,\J(R))$ is $\langle 0,0\rangle$. Symmetrically, the generic type of $N(\J(R),0)$ is $\langle 0,0\rangle$.
\end{proof}

\begin{remark}
The above lemma implies that if $R$ is a rank one uniserial domain and $K$ is a non-zero two-sided ideal then $N(K,\J(R))$ and $N(\J(R),K)$ are elementary equivalent.
\end{remark}

We now consider the points which are not topologically indistinguishable from critical or generic modules.

\begin{definition}
Let $R$ be a uniserial ring. Let $P,Q$ be two-sided non-zero completely prime ideals and $K$ a two-sided ideal such that $\langle P,K\rangle $ and $\langle K,Q\rangle $ are consistent. If there exists a point $N$ such that $N_{crit}=N(P,K)$ and $N_{gen}=N(K,Q)$ but $N$ is not in the closure of $N_{crit}$ or $N_{gen}$ then we call this point $(K,Q)\bot (P,K)$.

Note that this point is unique and isolated in its closure even before factoring out by $T_0$ by \ref{critgenalmostdet}. From the proof of \ref{critgenalmostdet} it may also be assumed that there exist $a,b\in R$ such that $(K,Q)\bot (P,K)\in \left(xb=0/a|x\right)$ but neither $N(K,Q)$ nor $N(P,K)$ are in $\left(xb=0/a|x\right)$.
\end{definition}

We now investigate when $(K,Q)\bot (P,K)$ exists.

\begin{lemma}
Let $R$ be a uniserial domain. Then $ (0,\J(R))\bot(\J(R),0)$ does not exist.
\end{lemma}
\begin{proof}
If $(0,\J(R))\bot(\J(R),0)$ exists then there exist $a,b\in R$ such that $(0,\J(R))\bot(\J(R),0)\in \left(xb=0/a|x\right)$ and $N(\J(R),0),N(0,\J(R))\notin \left(xb=0/a|x\right)$. Since the annihilator of $(\J(R),0)\bot(0,\J(R))$ is $0$ this implies that $ab=0$. Thus either $a=0$ or $b=0$. If $b=0$ then $N(0,\J(R))\in\left(xb=0/a|x\right)=\left(x=x/a|x\right)$. If $a=0$ then $N(\J(R),0)\in\left(xb=0/a|x\right)=\left(xb=0/x=0\right)$. Thus $(0,\J(R))\bot(\J(R),0)$ does not exist.
\end{proof}

\begin{lemma}
Let $R$ be a uniserial domain and $K$ a two-sided ideal. If $(K,\J(R))\bot(\J(R),K)$ exists then $K$ is finitely generated  as a right ideal.

\end{lemma}
\begin{proof}
If $N:=(K,\J(R))\bot(\J(R),K)$ then by definition there exists $a,b\in R$ such that $N\in \left(xb=0/a|x\right)$ and $N(K,\J(R)),N(\J(R),K)\notin \left(xb=0/a|x\right)$.

If $a\in K$ then $a|x$ is equivalent to $x=0$ in $N(K,\J(R))$ and $N(\J(R),K)$. Thus $N(\J(R),K)\in\left(xb=0/a|x\right)$. If $b\in K$ then $xb=0$ is equivalent to $x=x$ in $N(K,\J(R))$ and $N(\J(R),K)$. Thus $N(K,\J(R))\in \left(xb=0/a|x\right)$. Therefore $a,b\notin K$.

Since $N$ has annihilator $K$ and $N\in \left(xb=0/a|x\right)$, $ab\in K$.

Suppose that $c\in K$. Since $a\notin K$, $c=ab'$ for some $b'\in \J(R)$. Suppose for a contradiction that $b'\notin bR$. Then $b'r=b$ for some $r\in \J(R)$. Since $b\notin K$, $b'\notin K$. Thus $N(\J(R),K)\in \left(x=x/xb'=0\right)\cap\left(xr=0/x=0\right)=\left(xb=0/xb'=0\right)$. Since $ab'\in K$, $a|x$ implies $xb'=0$ in $N(\J(R),K)$. Thus $N(\J(R),K)\in \left(xb=0/a|x\right)$. Thus we have a contradiction and so, $b'\in bR$. Therefore $c=ab'\in abR$. Since $ab\in K$, $K=abR$.

%
%
%

\end{proof}

Note that over a rank one uniserial domain an ideal is finitely generated as a left ideal if and only if it is finitely generated as a right ideal by \ref{idealfgasleftideal}.

\begin{lemma}\label{JznRRznJexists}
Let $R$ be a rank one uniserial domain with $\J(R)^2=\J(R)$. Let $w\in R$ be such that $wR=Rw$ and $ab=w$.

\begin{enumerate}
\item $\langle bR,Ra\rangle$ is consistent.
\item The annihilator of $N(bR,Ra)$ is $wR$.
\item The point $(wR,\J(R))\bot(\J(R),Rw)$ exists and is $N(bR,Ra)$ where $ab=w$.
\end{enumerate}

\end{lemma}
\begin{proof}
$(1)$ Suppose $\alpha\notin Ra$, $\beta\notin bR$ and $\alpha\beta\in RabR=wR$. Since $\alpha\notin Ra$ and $\beta\notin bR$, $a=r\alpha$ and $b=\beta s$ for some $r,s\in\J(R)$. Since $\alpha\beta\in RabR=wR$, $\alpha\beta =w\lambda$ for some $\lambda\in R$. Thus $w=ab=r\alpha\beta s=rw\lambda s$. Since $wR=Rw$, this implies $w=r\mu w$ for some $\mu\in R$. This gives a contradiction, since $r\in \J(R)$. Thus, if $ab=w$ then, by \ref{Genasconsistencycondition}, $\langle bR,Ra\rangle$ is consistent.

\medskip
\noindent
$(2)$ Clearly $w\in \ann_RN(bR,Ra)$ because $N(bR,Ra)$ opens the pair $xb=0/a|x$ and so for all $m\in N(bR,Ra)$, $a|m$ implies $mb=0$.

If $r\notin Ra$ then $a=cr$ for some $c\in \J(R)$. Let $d=rb$. Then $w=ab=crb=cd$. Thus $\langle dR,Rc\rangle$ is a consistent pair and by \ref{equonepairs}, $N(bR,Ra)$ and $N(dR,Rc)$ are isomorphic. Thus $\ann_RN(bR,Ra)\subseteq rbR$ for all $r\notin Ra$. Since $\ann_RN(bR,Ra)$ is two-sided, by \ref{idealfgasleftideal}, $\ann_RN(bR,Ra)=rbR$ if and only if $\ann_RN(bR,Ra)=Rrb$. If $\ann_RN(bR,Ra)\subsetneq rbR$ then $rb\notin \ann_RN(bR,Ra)$ and hence $\ann_RN(bR,Ra)\subsetneq Rrb$. Thus $\ann_RN(bR,Ra)\subseteq Rrb$ for all $r\notin Ra$.

Suppose $\lambda\in \ann_RN(bR,Ra)$. Then $\lambda =\lambda' b$ for some $\lambda'\in R$. Suppose for a contradiction that $\lambda\notin Rw$.  Then $\lambda'\notin Ra$. So $a=c\lambda'$ for some $c\in \J(R)$. Since $\J(R)^2=\J(R)$, there exists $c_1,c_2\in \J(R)$ such that $c_1c_2=c$. Thus $c_2\lambda'\notin Ra$. But since $\lambda\in \ann_RN(bR,Ra)$, $\lambda \in Rc_2\lambda'b=Rc_2\lambda$. Since $R$ is a domain, this implies $1\in Rc_2$ and hence $c_2\notin \J(R)$. Thus we have a contradiction and so $\lambda\in wR=Rw$.

\medskip
\noindent
$(3)$ Since $\Ass N(Ra,bR)=\Div N(Ra,bR)=\J(R)$, the annihilator of $N(bR,Ra)$ is $wR$ and $N(bR,Ra)\in \left(xb=0/a|x\right)$ it is enough to show that $N(wR,\J(R)),N(\J(R),Rw)\notin \left(xb=0/a|x\right)$. Since $N(wR,\J(R))$ and $N(\J(R),Rw)$ are elementary equivalent, it is enough to show $N(wR,\J(R))\notin \left(xb=0/a|x\right)$. Note that $1+wR \in R/wR$ realises $(wR,\J(R))$.

Suppose $r\in R$ and $rb\in wR=Rw$.  Then $rb=sab$ for some $s\in R$. So $r=sa$. Thus the right module $R/wR$ does not open the pair $xb=0/a|x$. Therefore the pure-injective hull of $R/wR$ does not open the pair $xb=0/a|x$. So, since $N(bR,Ra)$ is a direct summand of the pure-injective hull of $R/wR$, $N(bR,Ra)\notin \left(xb=0/a|x\right)$.
\end{proof}

Thus we now know that any point which is topologically distinguishable from its generic and critical point is of the form $N(bR,Ra)$ where $abR=Rab$, that the critical type of such a point is $\langle \J(R),Rab\rangle$ and that the generic type of such a point is $\langle abR,\J(R)\rangle$. Note that we know from section \ref{pointsforuniserial} that these points are isolated in their closure and hence are $T_0$-points.

\begin{lemma}\label{listofpoints2}
Let $R$ be an exceptional rank one uniserial domain. All points in $\Zg_R$ are topologically indistinguishable from an indecomposable pure-injective of one of the following forms:
\begin{enumerate}
\item $N(0,0)$, $N(0,\J(R))$, $N(\J(R),0)$, $N(\J(R),\J(R))$
\item $N(\J(R),K)\equiv N(K,\J(R))$ where $K$ is a two-sided non-zero ideal
\item $N(bR,Ra)$ where $abR=Rab$
\end{enumerate}
\end{lemma}

We now describe as much of the topology of $\Zg_R$ as possible without specialising to $C_n$ for a particular $n\in\N$.

\begin{lemma}\label{pointwith2prinidealsisolated}
Let $R$ be an exceptional rank one uniserial domain with $\J(R)^2=\J(R)$. Let $w\in R$ be such that $wR=Rw$ and $ab=w$. The point $N(bR,Ra)$ is isolated in $\Zg_R$.
\end{lemma}
\begin{proof}
Since $R$ is exceptional, there exists a minimal ideal $I$ such that $wR\subsetneq I$. Take $\lambda\in I\backslash wR$. Then $N(bR,Ra)\in\left(x=x/x\lambda=0\right) $ and $N\in\left(x=x/x\lambda=0\right) $ if and only if $\ann_RN\subseteq wR$. Moreover $N(bR,Ra)\in\left(xb=0/a|x\right)$ and $N\in \left(xb=0/a|x\right)$ implies $w=ab\in\ann_RN$. Thus if $N\in \left(x=x/x\lambda=0\right)\cap\left(xb=0/a|x\right)$ then $\ann_RN=wR$. This means that $N$ is topologically indistinguishable from $N(bR,Ra)$ or $N(\J(R),Rw)\equiv N(wR,\J(R))$. We have already shown in the proof of \ref{JznRRznJexists} that $N(\J(R),Rw)\equiv N(wR,\J(R))\notin \left(xb=0/a|x\right)$. Thus $N\in \left(x=x/x\lambda=0\right)\cap\left(xb=0/a|x\right)$ implies $N$ is topologically indistinguishable from $N(bR,Ra)$. Since $N(bR,Ra)$ is isolated in its closure, $N=N(bR,Ra)$.
\end{proof}

\begin{lemma}\label{00andJJclosed}
Let $R$ be a rank one exceptional uniserial domain. The point $N(0,0)$ is closed and the point $N(\J(R),\J(R))$ closed and isolated. In particular, these points are $T_0$-points.
\end{lemma}
\begin{proof}
The type $\langle\J(R),\J(R)\rangle$ is realised in $R/\J(R)$ and the type $\langle 0,0\rangle$ is realised in $F$ the division ring of fractions of $R$. Both $R/\J(R)$ and $F$ are division rings and thus their Ziegler spectra have exactly one point $R/\J(R)$ and $F$ respectively. The canonical map from $R$ to $R/\J(R)$ is an epimorphism and the canonical embedding of $R$ into $F$ is an epimorphism. Thus, \cite[5.5.3]{PSL}, $\Zg_{R/\J(R)}$ and $\Zg_{F}$ embed into $\Zg_R$ as closed subsets. So $R/\J(R)$ and $F$ are closed points of $\Zg_R$.

We now show that $N(\J(R),\J(R))$ is isolated. Take $c\in \J(R)\backslash Q$ where $Q$ is the exceptional prime. Since $\J(R)^2=\J(R)$, there exist $a,b\in \J(R)$ such that $ab=c$. Thus $(\J(R),\J(R))\in \left(xb=0/a|x\right)$. If $N\in\left(xb=0/a|x\right)$ then $c=ab\in \ann_RN$. So $\ann_RN=\J(R)$. Thus $N=N(\J(R),\J(R))$.
\end{proof}

\begin{lemma}\label{nonfgannpointsiso}
Let $R$ be an exceptional rank one uniserial domain. If $K$ is a non-zero non-finitely generated ideal then $N(K,\J(R))\equiv N(\J(R),K)$ is isolated in $\Zg_R/T_0$.
\end{lemma}

\begin{proof}
Take $\lambda\in I_1\backslash K$ where $I_1$ is the smallest two-sided ideal strictly containing $K$. Then $N(K,\J(R))\in\left(x=x/x\lambda=0\right)$ since $\lambda\notin \ann_R N(\J(R),K)=K$. Take $a,\mu\in R$ such that $a\mu\in K\backslash I_2$ where $I_2$ is the largest two-sided ideal strictly contained in $K$, $\mu\in K$ and $a\in \J(R)$ (we can do this since $K$ is not finitely generated and hence $\J(R)K=K$). Then $N(\J(R),K)\in \left(x\mu=0/a|x\right)$.

If $N\in \left(x=x/x\lambda=0\right)$ then $\lambda\notin\ann_RN$. So $I_1\supseteq \lambda R\supsetneq \ann_RN$. If $N\in \left(x\mu=0/a|x\right)$ then $a\mu\in \ann_RN$. So $\ann_RN\supsetneq I_2$. Thus, if $N\in \left(x=x/x\lambda=0\right)\cap \left(x\mu=0/a|x\right)$ then $\ann_RN=K$.

Since $K$ is not finitely generated, the only point, up to topological indistinguishability, with annihilator $K$ is $N(K,\J(R))$. So we have shown that $N(K,\J(R))$ is isolated in $\Zg_R/T_0$.
\end{proof}

\begin{lemma}\label{closureofJ0and0J}
Let $R$ be a rank one uniserial domain. The closure, inside $\Zg_R/T_0$ of
\begin{enumerate}
\item $N(\J(R),0)$ is $\{N(0,0),N(\J(R),0)\}$
\item $N(0,\J(R))$ is $\{N(0,0),N(0,\J(R))\}$
\end{enumerate}
\end{lemma}
\begin{proof}
(1) Since $R$ is a domain, $\Div N(\J(R),0)=0$. So if $N$ is in the closure of $N(\J(R),0)$ then $\Div N=0$. Thus $\ann_R N=0$. If $\Ass N=0$ then $N=N(0,0)$. If $\Ass N=\J(R)$ then $N=N(\J(R),0)$. (2) follows by symmetry.
\end{proof}

\begin{lemma}\label{generalclosureofinfset}
Let $R$ be a rank one uniserial domain. Let $(N_n)_{n\in\N}$ be a sequence of points in $\Zg_R/T_0$ such that $\bigcap_{i=1}^\infty\ann_RN_i=\{0\}$ and for all $i\in\N$, $\Ass N_i=\J(R)$ and $\Div N_i=\J(R)$. The points $N(0,0), N(\J(R),0)$ and $N(0,\J(R))$ are in the closure of $\{N_i \st i\in\N\}$.
\end{lemma}
\begin{proof}
We need to show that if $N(0,\J(R))\in  U$ an open set then there exists $i\in\N$ such that $N_i\in  U$. It is enough to check this property for open sets in $U$ in the subbasis described in \ref{subbasis}.

Let $a,b,c,d,e\in R$. If $N(0,\J(R))\in\left(x=x/c|x\right)$ then $c\in \J(R)$. So $N_i\in \left(x=x/c|x\right)$ for all $i\in\N$ since $\Div N_i=\J(R)$. If $N(0,\J(R))\in (xd=0/x=0)$ then $d=0$. So $N_i\in \left(xd=0/x=0\right)$ for all $i\in\N$. If $N(0,\J(R))\in\left(xb=0/a|x\right)$ then $b=0$ and $a\in \J(R)$. Thus $N_i\in\left(xb=0/a|x\right)$ for all $i\in\N$ since $\Div N_i=\J(R)$. Finally, $N(0,\J(R))\in \left(x=x/xe=0\right)$ if and only if $e\neq 0$. Thus there exists $n\in\N$ such that $e\notin \ann_RN_n$ since $\bigcap_{i=1}^\infty\ann_RN_i=\{0\}$. Hence $N_n\in \left(x=x/xe=0\right)$. Therefore, the closure of a set of points $N_i$ as in the statement contains $N(\J(R),0)$. Symmetrically, one can show that the closure also contains $N(\J(R),0)$. By \ref{closureofJ0and0J}, the closure also contains $N(0,0)$.
\end{proof}

\subsection{Type $C_0$}

Throughout this subsection, let $R$ be an exceptional rank one uniserial domain of type $C_0$.

Using \ref{listofpoints2} and the fact that $R$ has no two-sided ideals which are finitely generated as a right ideal, we know that the points in $\Zg_R$ up to topological indistinguishability are $N(0,0)$, $N(\J(R),\J(R))$, $N(0,\J(R))$, $N(\J(R),0)$ and $X_n:=N(\J(R),Q^n)\equiv N(Q^n,\J(R))$ for $n\in\N$.

\begin{lemma}\label{XnclosedandisolatedC0}
The point $X_n$ is closed and isolated in $\Zg_R/T_0$.
\end{lemma}
\begin{proof}
That $X_n$ is isolated is a direct consequence of \ref{nonfgannpointsiso}.
%
%

If $N$ is in the closure of $X_n$ then $\ann_RN\supseteq Q^n$ by \ref{closure}. Thus $(\J(R),0), (0,\J(R))$ and $(0,0)$ are not in the closure of $X_n$. Since each $X_m$ is isolated, $X_m$ is not in the closure of $X_n$ for all $n\neq m$. Since $(\J(R),\J(R))$ is isolated, \ref{00andJJclosed}, $(\J(R),\J(R))$ is not in the closure of $X_n$. Thus the only point in the closure of $X_n$ is $X_n$ itself.
\end{proof}

\begin{proof}[proof of theorem \ref{typeC0}]
Suppose $C$ is a closed subset of $\Zg_R/T_0$. By \ref{closureofJ0and0J}, if $N(\J(R),0)\in C$ or $N(0,\J(R))\in C$ then $N(0,0)\in C$. So (a) of \ref{typeC0} holds. If $X_n\in C$ for infinitely many $n\in\N$ then since $\ann_RX_n=Q^n$, $\Ass X_n=\Div X_n=\J(R)$ and $\bigcap_{i=1}^\infty Q^n=0$, \ref{generalclosureofinfset} implies $N(\J(R),0),N(0,\J(R)),N(0,0)\in C$. Thus (b) of \ref{typeC0} also holds.

Now suppose that $C\subseteq \Zg_R/T_0$ satisfies (a) and (b). First note that, by \ref{closureofJ0and0J}, \ref{00andJJclosed} and \ref{XnclosedandisolatedC0}, $C$ is closed under specialisation. We intend to show that $C$ is equal to its closure.

Suppose $N(\J(R),0)\notin C$. Since (b) holds, $\{X_n \st X_n\in C\}$ is finite. So $C$ is finite. Since $C$ is closed under specialisation, $C=\cl C$. The same argument shows that if $N(0,\J(R))\notin C$ or $N(0,0)\notin C$ then $C=\cl C$.

Now suppose that $N(\J(R),0), N(0,\J(R)), N(0,0)\in C$. Since $N(\J(R),\J(R))$ is isolated, $N(\J(R),\J(R))\in C$ if and only if $N(\J(R),\J(R))\in\cl C$. Since $X_n$ is isolated for each $n\in\N$, $X_n\in C$ if and only if $X_n\in \cl C$. Thus $C=\cl C$.

\end{proof}

\subsection{Type $C_k$ for $k>1$}

Using \ref{listofpoints2}, we know that the points in $\Zg_R$ up to topological indistinguishability are $N(0,0), N(\J(R),\J(R)), N(0,\J(R)), N(\J(R),0)$, $X_n:=N(\J(R),Q^n)\equiv N(Q^n,\J(R))$, $Z_n:=N(z^nR,\J(R))\equiv N(\J(R),Rz^n)$ and $Y_{n}:=N(bR,Ra)$ where $ab=z^n$ for $n\in\N$. By \ref{JznRRznJexists}, $(\J(R),Rz^{n})\bot(z^{n}R,\J(R))$ is $Y_{n}$.

\begin{lemma}\label{XnisolatedCn}
The points $X_n$ and $Y_n$ are isolated in $\Zg_R/T_0$ for all $n\in\N$.
\end{lemma}
\begin{proof}
The point $X_n$ is isolated by \ref{nonfgannpointsiso} and the point $Y_n$ is isolated by \ref{pointwith2prinidealsisolated}.
%
%

\end{proof}

\begin{lemma}\label{xknynznopen}
The set $\{X_{kn},Y_n,Z_n\}$ is open for all $n,k\in \N$.
\end{lemma}
\begin{proof}
 Note that both $Y_n$ and $Z_n$ have annihilator $z^nR=Rz^n$ and $X_{kn}$ has annihilator $Q^{kn}=z^n\J(R)$. Take $\lambda\in Q^{kn-1}\backslash z^nR$. Then $N\in \left(x=x/x\lambda=0\right)$ if and only if $\lambda\notin \ann_RN$ if and only if $Q^{kn-1}\supsetneq \ann_RN$.  Thus $X_{kn},Y_n,Z_n\in \left(x=x/x\lambda=0\right)$.

Let $a,b\in \J(R)$ be such that $ab=z^n$. Take $c,\mu\in \J(R)$ such that $cz^n\mu\in z^n\J(R)\backslash Q^{kn+1}$. We can do this because $z^n\J(R)=\J(R)z^n$ and $\J(R)^2=\J(R)$. Since $N(\J(R),Rz^n)=N(b\J(R),Ra)$, $Z_n\in \left(xb\mu=0/ca|x\right)$. It is clear that $Y_n\in \left(xb\mu=0/ca|x\right)$. Since $Q^{kn}=\J(R)z^n$, $N(\J(R),Q^{kn})=N(b\J(R),\J(R)a)$ and so $X_{kn}\in \left(xb\mu=0/ca|x\right)$.

Moreover, if $N\in\left(xb\mu=0/ca|x\right)$ then $cz^n\mu\in\ann_RN$. So $\ann_RN\supsetneq Q^{kn+1}$.

Thus if $N\in \left(x=x/x\lambda=0\right)\cap\left(xb\mu=0/ca|x\right)$ then $Q^{kn-1}\supsetneq\ann_RN\supsetneq Q^{kn+1}$. Thus $\ann_RN= z^nR$ or $\ann_RN=z^n\J(R)=Q^{kn}$. Thus $\left(x=x/x\lambda=0\right)\cap\left(xb\mu=0/ca|x\right)=\{X_{kn},Y_n,Z_n\}$.


\end{proof}

\begin{lemma}\label{clXkn}
The closure of $X_{kn}$ is $\{X_{kn},Z_n\}$.
\end{lemma}
\begin{proof}
In order to check that $Z_n$ is in the closure of $X_{kn}$ it is enough to check for $(U_i)_{i\in I}$ an open neighbourhood basis of $Z_n$, $X_{kn}\in U_i$ for all $i\in I$.

Since all pp-$1$-formulae are equivalent to a formula of the form $\sum_{i=1}^nxa_i=0\wedge b_i|x$ for some $a_i,b_i\in R$ and $n\in\N$, Applying Prest's duality \cite[chapter 8]{Mikebook1}, this means all pp-$1$-formulae are equivalent to a formula of the form $\bigwedge_{i=1}^nxc_i=0+d_i|x$. Since the pp-definable subgroups of any $N\in\Zg_R$ are totally ordered by inclusion, \cite[4.9]{Zieglermodules} implies that the open sets $\left(xa=0\wedge b|x/xc=0+d|x\right)$, where $a\in z^nR$, $b\notin \J(R)$, $c\notin z^nR$ and $d\in\J(R)$, are a basis of open neighbourhoods for $Z_n$.

Suppose $a\in z^nR$, $b\notin \J(R)$, $c\notin z^nR$ and $d\in\J(R)$, then
\[\left(xa=0\wedge b|x/xc=0+d|x\right)=\left(xa=0/xc=0+d|x\right)=\left(xa=0/xc=0\right)\cap\left(xa=0/d|x\right).\]

Since $a\in z^nR$ and $c\notin z^nR$ there exists $r\in \J(R)$ such that $cr=a$ and $N\in\left(xa=0/xc=0\right)$ if and only if $r\in\Ass N$ and $c\notin \ann_RN$. Since $\Ass X_{kn}=\Ass Z_n$ and $\ann_R X_{kn}\subsetneq \ann_R Z_n$, $X_{kn}\in \left(xa=0/xc=0\right)$.

We now show that $X_{kn}\in \left(xa=0/d|x\right)$. If $a\in z^n\J(R)$ then we are done since $d\in \J(R)$. Otherwise $a=z^n\mu$ for some $\mu\notin\J(R)$ and hence $xa=0$ is equivalent to $xz^n=0$. So we need to show that $Z_{kn}\in \left(xz^n=0/d|x\right)$ for all $d\in \J(R)$.

Take $\lambda\in \J(R)\backslash (Q\cup Rd)$ and let $m$ realise $\langle z^n\J(R),\J(R)\rangle$ in $X_{kn}$. Then $m\lambda z^n=0$ since $\lambda\in \J(R)$ and hence $\lambda z^n\in z^n\J(R)=\J(R)z^n$. Since $\lambda\notin Rd$ there exists $\delta\in \J(R)$ such that $d=\delta \lambda$. If $d|m\lambda$ then $m$ is in the pp-definable subgroup defined by $x\lambda=0+\delta|x$. Since the pp-definable subgroups of $X_{kn}$ are totally ordered, this would mean that $m\lambda=0$ or $\delta|m$. This is false because $\lambda\notin z^n\J(R)\subseteq Q$ and $\delta\notin \J(R)$. Thus $m\lambda$ opens $xz^n=0/d|x$. So $X_{kn}\in \left(xz^n=0/d|x\right)$.

Thus we have shown that $Z_n$ is in the closure of $X_{kn}$.

%
%

We now show that only $X_{kn}$ and $Z_n$ are in the closure of $X_{kn}$. Since for each $m\in\N$, $Y_m$ is isolated, $Y_m$ is not in the closure of $X_{kn}$. Since for all $kn\neq m\in\N$, $X_m$ is isolated, $X_m$ is not in the closure of $X_{kn}$. Since $N(\J(R),\J(R))$ is isolated, $N(\J(R),\J(R))$ is not in the closure of $X_{kn}$.

Since for all $m\in\N$, $\{X_{km},Y_m,Z_m\}$ is open, $Z_m$ is not in the closure of $X_{kn}$ unless $m=kn$.

Finally note that $N(0,0)$ is not in the closure of $X_{kn}$ since for all $\lambda\neq 0$, $N(0,0)\in \left(x=x/x\lambda=0\right)$. Thus $N(\J(R),0)$ and $N(0,\J(R))$ are not in the closure of $X_{kn}$.

\end{proof}

\begin{lemma}\label{clYn}
The closure of $Y_n$ is $\{Y_n,Z_n\}$.
\end{lemma}
\begin{proof}
Since $Y_n$ is $(\J(R),Rz^{n})\bot(z^{n}R,\J(R))$, $Z_n$ is in the closure of $Y_n$. Since for each $m$, $X_m$ is isolated, $X_m$ is not in the closure of $Y_n$. Since for each $m$, $\{X_{km},Y_m,Z_m\}$ is open $Z_m$ and $Y_m$ are not in the closure of $Y_n$ unless $m=n$.
Finally note that $N(0,0)$ is not in the closure of $Y_{n}$ since for all $\lambda\neq 0$, $N(0,0)\in \left(x=x/x\lambda=0\right)$. Thus $N(\J(R),0)$ and $N(0,\J(R))$ are not in the closure of $Y_n$.
\end{proof}

\begin{proof}[Proof of theorem \ref{typeCk}]
Suppose that $C$ is closed. Thus $C$ is closed under specialisation and hence (a), (b) and (c) hold by \ref{clXkn}, \ref{clYn} and \ref{closureofJ0and0J}. Since any infinite set of points of the form $X_n$, $Y_n$ or $Z_n$ has common annihilator zero, \ref{generalclosureofinfset} implies (d) holds.

Suppose (a),(b), (c), (d) hold for $C\subseteq Zg_R/T_0$. Properties (a), (b) and (c) imply that $C$ is closed under specialisation. We intend to show that $C$ is equal to its closure. If $N(\J(R),0)\notin C$, $N(0,\J(R))\notin C$ or $N(0,0)\notin C$ then $C$ is finite by (d). Thus $C$ is equal to its closure since $C$ is closed under specialisation. Now suppose that $N(0,\J(R)), N(\J(R),0), N(0,0)\in C$. If $N$ is isolated then $N\in C$ if and only if $N\in \cl C$. Thus we need only concern ourselves with non-isolated points. So, suppose that $Z_n\in \cl C$. Since $\{Y_n,X_{kn},Z_n\}$ is open, $C$ contains either $Z_n,Y_n$ or $X_{kn}$. Thus (a) and (b) imply $Z_n\in C$. Thus $C$ is equal to its closure and hence closed.

\end{proof}

\subsection{Type $C_1$}\label{ssTypeC1}

The only difference between the $C_1$ case and the $C_k$ case for $k>1$ is that the points are different. For this reason, most proofs will not be given.

Using \ref{listofpoints2}, we know that the points in $\Zg_R$ up to topological indistinguishability are $N(0,0)$, $N(\J(R),\J(R))$, $N(0,\J(R))$, $N(\J(R),0)$, $X_n:=N(\J(R),\J(R)z^n)\equiv N(z^n\J(R),\J(R))$, $Z_n:=N(z^nR,\J(R))\equiv N(\J(R),Rz^n)$ and $Y_{n}:=N(bR,Ra)$ where $ab=z^n$. By \ref{JznRRznJexists},  $Y_{n}$ is $(\J(R),Rz^{n})\bot(z^{n}R,\J(R))$.

\begin{lemma}
The points $X_n$ and $Y_n$ are isolated in $\Zg_R/T_0$.
\end{lemma}
\begin{proof}That $X_n$ is isolated follows from \ref{nonfgannpointsiso}.
That $Y_n$ is isolated is exactly \ref{pointwith2prinidealsisolated}.

%
%
%
\end{proof}

\begin{lemma}
The set $\{X_n,Y_n,Z_n\}$ is open for all $n\in\N$.
\end{lemma}
\begin{proof}
The proof here is exactly as in \ref{xknynznopen}.
\end{proof}

\begin{lemma}\label{specialisationC1}
The closure of $X_n$ is $\{X_n,Z_n\}$ and the closure of $Y_n$ is $\{Y_n,Z_n\}$. Thus $Z_n$ is a closed point.
\end{lemma}
\begin{proof}
This is exactly as in \ref{clXkn} and \ref{clYn}.
%
%
\end{proof}

The proof of \ref{typeC1} is now exactly as in \ref{typeCk}.

\bibliographystyle{alpha}
\bibliography{serial}

\end{document}

%% file: macros.tex
\usepackage{amsthm}
\usepackage{amssymb}
\usepackage{amsmath}
\usepackage[shortlabels]{enumitem}

\usepackage{ifpdf}
\ifpdf
\usepackage[pdftex,linktocpage]{hyperref}
\else
\usepackage[hypertex,linktocpage]{hyperref}
\fi

\input xy
\xyoption{all} 

\newtheorem{theorem}{Theorem}

\newtheorem{definition}[theorem]{Definition}
\newtheorem{lemma}[theorem]{Lemma}

\newtheorem{proposition}[theorem]{Proposition}
\newtheorem{remark}[theorem]{Remark}
\newtheorem{cor}[theorem]{Corollary}

\newtheorem{question}{Question}

\numberwithin{theorem}{section}

\renewcommand\iff{%
\ifmmode\text{ if and only if }%
\else if and only if \fi}

\renewcommand{\and}{\wedge}

\renewcommand{\phi}{\varphi}

\newcommand{\Ass}{\textnormal{Ass}}

\newcommand{\End}{\textnormal{End}}

\newcommand{\ann}{\textnormal{ann}}

\newcommand{\zg}{\textnormal{Zg}}
\newcommand{\Zg}{\textnormal{Zg}}

\newcommand{\pinj}{\textnormal{pinj}}
\newcommand{\cl}{\textnormal{cl}}

\newcommand{\mcal}[1]{\mathcal{#1}}

\newcommand{\mfrak}[1]{\mathfrak{#1}}
\newcommand{\st}{\ \vert \ }
\newcommand{\vertl}{\left\vert}
\newcommand{\vertr}{\right\vert}

\newcommand{\pp}{\textnormal{pp}}

\newcommand{\R}{\mathbb{R}}

\newcommand{\Q}{\mathbb{Q}}
\newcommand{\N}{\mathbb{N}}
\newcommand{\Z}{\mathbb{Z}}

\newcommand{\J}{\textnormal{J}}
\newcommand{\PE}{\textnormal{PE}}